\documentclass[a4paper,oneside,11pt]{article}

\usepackage[a4paper]{geometry}
\usepackage{aeguill}                
\usepackage{graphicx}
\usepackage{amsmath,amsfonts,amssymb,amsthm}
\usepackage{pstricks} 
\usepackage{epstopdf,comment,enumerate}
\usepackage{srcltx}
\usepackage[utf8]{inputenc} 
\usepackage[T1]{fontenc} 
 
\usepackage[english]{babel}

\title{A compactification of outer space which is an absolute retract}
\author{Mladen Bestvina and Camille Horbez\thanks{The first author gratefully acknowledges the support by the National Science Foundation under the grant number DMS-1308178. The second author is grateful to the University of Utah for its hospitality.}}

\usepackage[all]{xy}
\usepackage{bbm}
\begin{document}
\maketitle
\newtheorem{de}{Definition} [section]
\newtheorem{theo}[de]{Theorem} 
\newtheorem{prop}[de]{Proposition}
\newtheorem{lemma}[de]{Lemma}
\newtheorem{cor}[de]{Corollary}
\newtheorem{propd}[de]{Proposition-Definition}

\theoremstyle{remark}
\newtheorem{rk}[de]{Remark}
\newtheorem{ex}[de]{Example}
\newtheorem{question}[de]{Question}

\normalsize

\addtolength\topmargin{-.5in}
\addtolength\textheight{1.in}
\addtolength\oddsidemargin{-.045\textwidth}
\addtolength\textwidth{.09\textwidth}

\newcommand{\imp}{\Rightarrow}
\newcommand{\ra}{\rightarrow}
\newcommand{\m}{^{-1}}
\newcommand{\dunion}{\sqcup}
\newcommand{\eps}{\varepsilon}
\renewcommand{\epsilon}{\varepsilon}
\newcommand{\calf}{\mathcal{F}}
\newcommand{\caly}{\mathcal{Y}}
\newcommand{\calb}{\mathcal{B}}
\newcommand{\calq}{\mathcal{Q}}
\newcommand{\Z}{\mathcal{Z}}
\newcommand{\Oo}{\mathcal{O}}
\newcommand{\calu}{\mathcal{U}}
\newcommand{\calv}{\mathcal{V}}
\newcommand{\calw}{\mathcal{W}}
\newcommand{\bbR}{\mathbb{R}}
\newcommand{\bbZ}{\mathbb{Z}}
\newcommand{\actson}{\curvearrowright}
\newcommand{\es}{\emptyset}
\newcommand{\grp}[1]{\langle #1 \rangle}
\newcommand{\dg}{\dagger}
\newcommand{\Char}{\text{Char}}

\def\R{{\mathbb R}}
\makeatletter
\edef\@tempa#1#2{\def#1{\mathaccent\string"\noexpand\accentclass@#2 }}
\@tempa\rond{017}
\makeatother

\begin{abstract}
We define a new compactification of outer space $CV_N$ (the
\emph{Pacman compactification}) which is an absolute retract, for
which the boundary is a $Z$-set. The classical compactification
$\overline{CV_N}$ made of very small $F_N$-actions on
$\mathbb{R}$-trees, however, fails to be locally $4$-connected as soon
as $N\ge 4$. The Pacman compactification is a blow-up of
$\overline{CV_N}$, obtained by assigning an orientation to every arc
with nontrivial stabilizer in the trees.
\end{abstract}

\section*{Introduction}

The study of the group $\text{Out}(F_N)$ of outer automorphisms of a
finitely generated free group has greatly benefited from the study of
its action on Culler--Vogtmann's outer space $CV_N$. It is therefore
reasonable to look for compactifications of $CV_N$ that have ``nice"
topological properties. The goal of the present paper is to construct
a compactification of $CV_N$ which is a compact, contractible,
finite-dimensional absolute neighborhood retract (ANR), for which the
boundary is a Z-set.

One motivation for finding ``nice" actions of a group $G$ on absolute
retracts comes from the problem of solving the Farrell--Jones
conjecture for $G$, see \cite[Section 1]{BLR08}. For instance, it was
proved in \cite{BM91} that the union of the Rips complex of a
hyperbolic group together with the Gromov boundary is a compact, contractible
ANR, and this turned out to be a crucial ingredient in the proof by
Bartels--Lück--Reich of the Farrell--Jones conjecture for hyperbolic
groups \cite{BLR08}. A similar approach was recently used by Bartels
to extend these results to the context of relatively hyperbolic groups
\cite{Bar15}, and by Bartels--Bestvina to the context of mapping class
groups \cite{BB}.
\\
\\
\indent We review some terminology. A compact metrizable space $X$ is said to
be an \emph{absolute (neighborhood) retract} (AR or ANR) if for every compact
metrizable space $Y$ that contains $X$ as a closed subset, the
space $X$ is a (neighborhood) retract of $Y$. Given $x\in X$, we say that $X$ is \emph{locally contractible ($LC$) at $x$} if for every open
neighborhood $\calu$ of $x$, there exists an open neighborhood
$\calv\subseteq\calu$ of $x$ such that the inclusion map
$\calv\hookrightarrow\calu$ is nullhomotopic. More generally, $X$ is
{\it locally $n$-connected ($LC^n$) at $x$} if for every open
neighborhood $\calu$ of $x$, there exists an open neighborhood
$\calv\subseteq\calu$ of $x$ such that for every $0\leq i\leq n$,
every continuous map $f:S^i\to \calv$ from the $i$-sphere is nullhomotopic in
$\calu$. We then say that $X$ is $LC$ (or $LC^n$) if it is $LC$ (or $LC^n$) at every point $x\in X$.

A nowhere dense closed subset $Z$ of a compact metrizable space $X$ is a \emph{Z-set} if $X$ can be instantaneously
homotoped off of $Z$, i.e. if there exists a homotopy $H:X\times
[0,1]\to X$ so that $H(x,0)=x$ and $H(X\times(0,1])\subseteq
  X\setminus Z$. Given $z\in Z$, we say that $Z$ is \emph{locally complementarily contractible ($LCC$) at $z$}, resp.\ \emph{locally complementarily $n$-connected ($LCC^n$) at $z$}, if for every open neighborhood $\calu$ of $z$ in $X$, there exists a smaller open neighborhood $\calv\subseteq\calu$ of $z$ in $X$ such that the inclusion $\calv\setminus Z\hookrightarrow\calu\setminus Z$ is nullhomotopic in $X$, resp.\ trivial in $\pi_i$ for all $0\le i\le n$. We then say that $Z$ is $LCC$ (resp.\ $LCC^n$) in $X$ if it is $LCC$ (resp.\ $LCC^n$) at every point $z\in Z$. 
  
Every ANR space is locally contractible. Further, if $X$ is an ANR, and $Z\subseteq X$ is a Z-set, then $Z$ is $LCC$ in $X$. Conversely, it is a classical fact that every finite-dimensional, compact, metrizable, locally contractible space $X$ is an ANR, and further, an $n$-dimensional $LC^n$ compact metrizable
space is an ANR, see \cite[Theorem V.7.1]{hu}. If $X$ is further assumed to
be contractible then $X$ is an AR. In Appendix \ref{sec-miracle} of the present paper, we will establish (by similar methods) a slight generalization of this fact, showing that if $X$ is an $n$-dimensional compact metrizable space, and $Z\subseteq X$ is a nowhere dense closed subset which is $LCC^n$ in $X$, and such that $X\setminus Z$ is $LC^n$, then $X$ is an ANR and $Z$ is a Z-set in $X$.
\\ 
\\ 
\indent Culler--Vogtmann's \emph{outer
    space} $CV_N$ can be defined as the space of all $F_N$-equivariant
  homothety classes of free, minimal, simplicial, isometric
  $F_N$-actions on simplicial metric trees (with no valence $2$
  vertices).  Culler--Morgan's compactification of outer space \cite{CM87} can be described by
  taking the closure in the space of all $F_N$-equivariant homothety
  classes of minimal, nontrivial $F_N$-actions on $\mathbb{R}$-trees,
  equipped with the equivariant Gromov--Hausdorff topology. The
  closure $\overline{CV_N}$ identifies with the space of homothety
  classes of minimal, \emph{very small} $F_N$-trees
  \cite{CL95,BF94,Hor14-2}, i.e.\ those trees whose arc stabilizers are
  cyclic and root-closed (possibly trivial), and whose tripod stabilizers are trivial. Outer space $CV_N$ is contractible and locally contractible \cite{CV86}.

When $N=2$, the closure $\overline{CV_2}$ was completely described by
Culler--Vogtmann in \cite{CV91}. The closure of reduced outer space
(where one does not allow for separating edges in the quotient graphs)
is represented on the left side of Figure \ref{fig-cv2}: points in the circle at
infinity represent actions dual to measured foliations on a
once-punctured torus, and there are ``spikes" coming out corresponding
to simplicial actions where some edges have nontrivial
stabilizer. These spikes prevent the boundary
$\overline{CV_2}\setminus CV_2$ from being a Z-set in
$\overline{CV_2}$: these are locally separating subspaces in
$\overline{CV_2}$, and therefore $\overline{CV_2}$ is not LCC (it is not even $LCC^0$) at points on these spikes. More surprisingly, while $\overline{CV_2}$ is an absolute retract (and we believe that so is
$\overline{CV_3}$), this property
fails as soon as $N\ge 4$.

\begin{figure}
\begin{center}
\includegraphics[width=.3\textwidth]{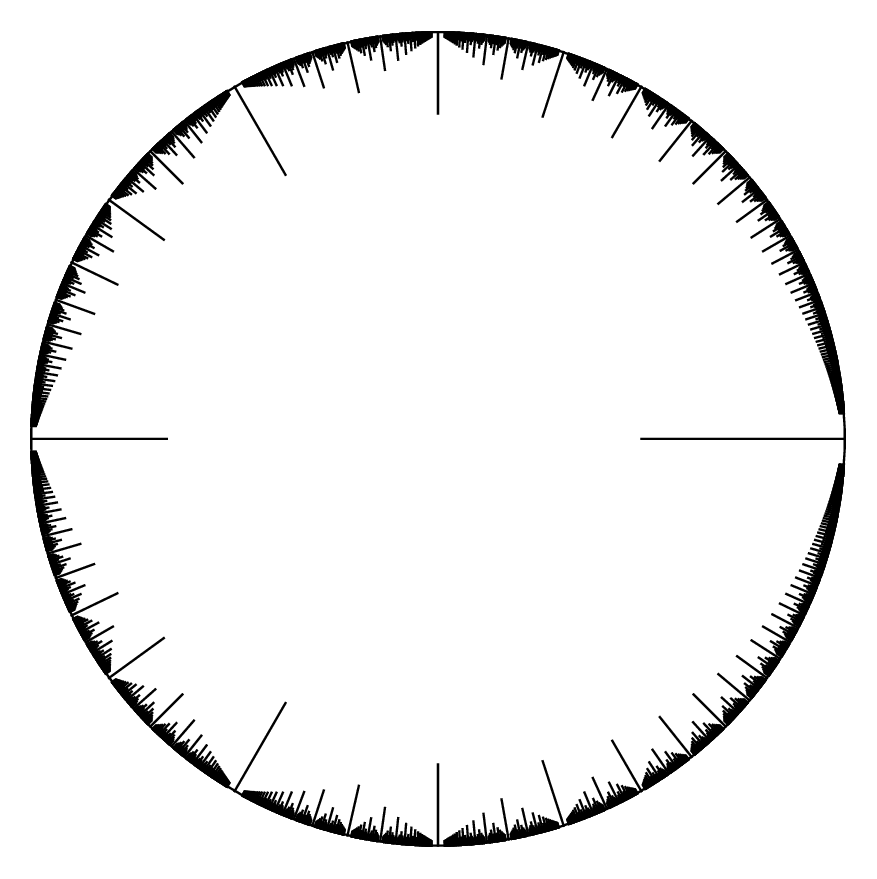}
~~~~~~~~~~~~
\includegraphics[width=.3\textwidth]{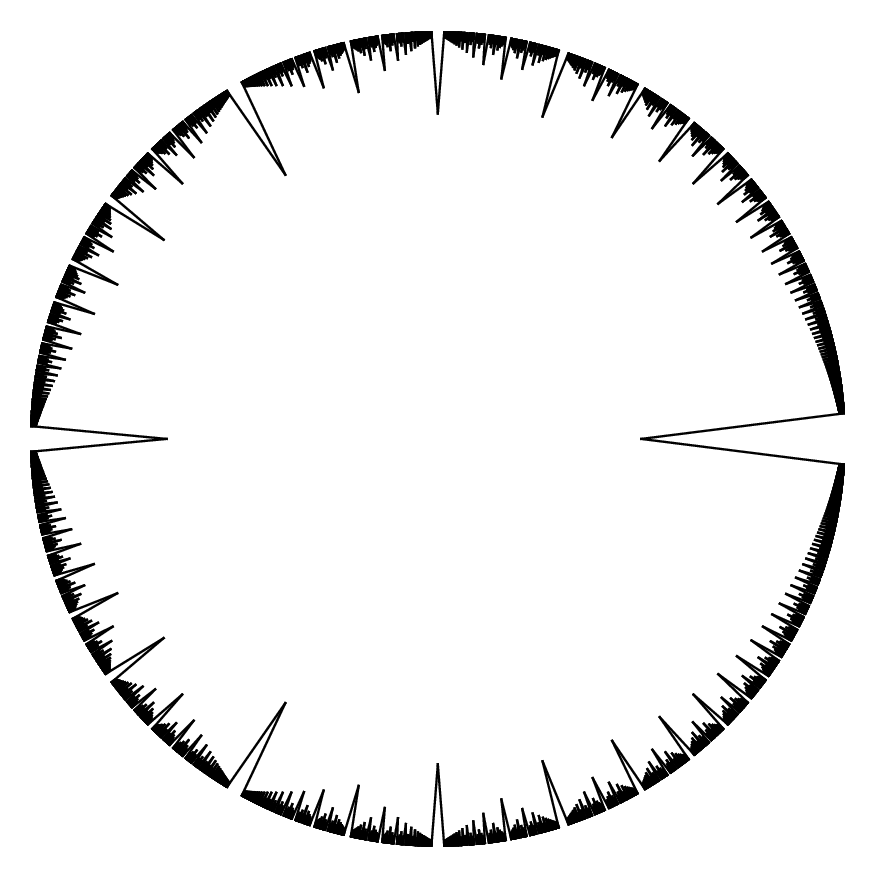}
\caption{The reduced parts of the classical compactification $\overline{CV_2}$ (on the left), and of the Pacman compactification $\widehat{CV_2}$ (on the right).}
\label{fig-cv2}
\end{center}
\end{figure}

\theoremstyle{plain}
\newtheorem*{theo:1}{Theorem 1}

\begin{theo:1}\label{notar}
For all $N\ge 4$, the space $\overline{CV_N}$ is not locally $4$-connected, hence it is not an AR.
\end{theo:1} 

There are however many trees in $\overline{CV_N}$ at which $\overline{CV_N}$ is locally contractible: for example, we prove in Section \ref{sec-triv} that $\overline{CV_N}$ is locally contractible at any tree with trivial arc stabilizers. The reason why local $4$-connectedness fails in general is the following. When $N\ge 4$, there are trees in $\overline{CV_N}$ that contain
both a subtree dual to an arational measured foliation on a nonorientable
surface $\Sigma$ 
of genus $3$ with a single boundary curve 
$c$, and a simplicial edge with nontrivial stabilizer $c$. We construct such a tree $T_0$ (see Section \ref{sec-not-ar} for its precise definition), at which $\overline{CV_N}$ fails to be locally $4$-connected, due to the combination of the
following two phenomena.
\begin{itemize}
\item The space $X(c)$ of trees where $c$ fixes a nondegenerate arc locally
  separates $\overline{CV_N}$ at $T_0$. \item The subspace of
  $\mathcal{PMF}(\Sigma)$ made of foliations which are dual to very
  small $F_N$-trees contains arbitrarily small embedded $3$-spheres which are not
  nullhomologous: these arise as $\mathcal{PMF}(\Sigma')$ for some
  orientable subsurface $\Sigma'\subseteq\Sigma$ which is the
  complement of a Möbius band in $\Sigma$. Notice here that a tree dual to a geodesic curve on $\Sigma$ may fail to be
very small, in the case where the curve is one-sided in $\Sigma$.
\end{itemize} 
\noindent We will find an open neighborhood $\calu$ of $T_0$ in $\overline{CV_N}$ such that for any smaller neighborhood $\calv\subseteq\calu$ of $T_0$, we can find a $3$-sphere $S^3$ in $X(c)\cap\calv$ (provided by the second point above) which is not nullhomologous in $X(c)\cap\calu$, but which can be capped off by balls $B^4_{\pm}$ in each of the two complementary components of $X(c)\cap\calv$ in $\calv$. By gluing these two balls along their common boundary $S^3$, we obtain a $4$-sphere in $\calv$, which is shown not to be nullhomotopic within $\calu$ by appealing to a \v{C}ech homology argument, presented in Appendix~\ref{sec-Cech} of the paper. 

It would be of interest to have a better
understanding of the topology of the space $\mathcal{PMF}(\Sigma)$ in
order to have a precise understanding of the failure of local
connectivity of $\overline{CV_N}$.   
\\
\\
\textit{Question:} Is $\overline{CV_N}$ locally 3-connected for every $N$? Is $\overline{CV_3}$ locally contractible?
\\
\\
\indent However, we also build a new 
compactification $\widehat{CV_N}$ of $CV_N$ (a blow-up of
$\overline{CV_N}$) which is an absolute retract, for which the
boundary is a $Z$-set. The remedy to the bad phenomena described above
is to prescribe orientations in an $F_N$-equivariant way to all arcs
with nontrivial stabilizers in trees in $\overline{CV_N}$, which has the effect in particular to ``open up" $\overline{CV_N}$ at the problematic spaces $X(c)$. In other
words, the characteristic set of any element $g\in
F_N\setminus\{e\}$ in a tree $T$ is given an orientation as soon
as it is not reduced to a point: when $g$ acts as a hyperbolic
isometry of $T$, its axis comes with a natural orientation, and we
also decide to orient the edges with nontrivial stabilizers. Precise
definitions of $\widehat{CV_N}$ and its topology are given in Section
\ref{sec-pacman} of the present paper. In rank $2$, this operation has the
effect of ``cutting" along the spikes (see Figure \ref{fig-cv2}),
which leads us to call this new compactification the \emph{Pacman
  compactification} of outer space.

\theoremstyle{plain}
\newtheorem*{theo:2}{Theorem 2}

\begin{theo:2}\label{ar}
The space $\widehat{CV_N}$ is an absolute retract of dimension $3N-4$,
and $\widehat{CV_N}\setminus CV_N$ is a $Z$-set.
\end{theo:2}

The space $\widehat{CV_N}$ is again compact, metrizable and
finite-dimensional: this is established in Section \ref{sec-pacman} of the
present paper from the analogous results for $\overline{CV_N}$. Also, we show in Section \ref{sec-ar} that 
every point in $\widehat{CV_N}$ is a limit of points in $CV_N$. The
crucial point for proving Theorem~2 is to show that the boundary $\widehat{CV_N}\setminus CV_N$ is locally complementarily contractible.

The proof of this last fact is by induction on the rank $N$, and the
strategy is the following. Given a tree $T\in\widehat{CV_N}\setminus CV_N$,
one can first approximate $T$ by trees that split as graphs of actions
over free splittings of $F_N$, and admit $1$-Lipschitz
$F_N$-equivariant maps to $T$. Using our induction hypothesis (and
working in the outer space of each of the factors that are elliptic in the splitting), we
prove that subspaces in $\widehat{CV_N}$
made of trees that split as graphs of actions over a given free splitting are locally complementarily contractible at every point in the boundary. We also find a continuous way of deforming a neighborhood of $T$ into one of these subspaces, so that $T$ is sent to a nearby tree. This
enables us to prove that $\widehat{CV_N}\setminus CV_N$ is locally complementarily contractible at $T$.

\paragraph*{Acknowledgments.} We would like to thank the anonymous referee for their numerous suggestions to improve the exposition of the paper.

\setcounter{tocdepth}{1}
\tableofcontents

\section{The space $\overline{CV_N}$ is not an AR when $N\ge 4$.}\label{sec-not-ar}

\subsection{Review: Outer space and Culler--Morgan's compactification}\label{sec-def}

\paragraph*{Outer space and its closure.} Let $N\ge 2$. \emph{Outer space} $CV_N$ (resp. \emph{unprojectivized
  outer space} $cv_N$) is the space of $F_N$-equivariant homothety
(resp.\ isometry) classes of simplicial, free, minimal, isometric
$F_N$-actions on simplicial metric trees, with no valence $2$
vertices. Here we recall that an $F_N$-tree is \emph{minimal} if it does not contain any proper $F_N$-invariant subtree. Unprojectivized outer space can be embedded into the space
of all $F_N$-equivariant isometry classes of minimal $F_N$-actions on
$\mathbb{R}$-trees, which is equipped with the \emph{equivariant Gromov--Hausdorff
  topology} introduced in \cite{Pau88,Pau89}. This is the
topology for which a basis of open neighborhoods of a tree $T$ is
given by the sets $\mathcal{N}_T(K,X,\epsilon)$ (where $K\subseteq T$
is a finite set of points, $X\subseteq F_N$ is a finite subset, and
$\epsilon>0$), defined in the following way: an $F_N$-tree $T'$ belongs
to $\mathcal{N}_T(K,X,\epsilon)$ if there exists a finite set
$K'\subseteq T'$ and a bijection $K\to K'$ such that for all $x,y\in
K$ and all $g\in X$, one has $|d_{T'}(x',gy')-d_{T}(x,gy)|<\epsilon$
(where $x',y'$ are the images of $x,y$ under the bijection). The
closure $\overline{cv_N}$ was identified in \cite{CL95,BF94,Hor14-2}
with the space of $F_N$-equivariant isometry classes of minimal, very
small actions of $F_N$ on $\mathbb{R}$-trees (an action is called
\emph{very small} if arc stabilizers are cyclic and root-closed [possibly trivial], and
tripod stabilizers are trivial). Note that we allow for the trivial
action of $F_N$ on a point in $\overline{cv_N}$. The compactification
$\overline{CV_N}$ is the space of homothety classes of nontrivial
actions in $\overline{cv_N}$. In the present paper, for carrying induction arguments, we will need to allow for the case where $N=1$, in which case $cv_1$ is the collection of all possible isometry classes of $\mathbb{Z}$-actions on the real line (these are just parameterized by the translation length of the generator), and $\overline{cv_1}$ is obtained by adding the trivial action on a point. 

\paragraph*{Structure of the trees in $\overline{cv_N}$: the Levitt decomposition.} A \emph{splitting} of $F_N$ is a minimal, simplicial $F_N$-tree. A tree $T\in\overline{cv_N}$ is said to \emph{split as a graph of actions} over a splitting $S$ if there exist 
\begin{itemize}
\item for each vertex $u$ of $S$ with stabilizer $G_u$, a $G_u$-tree $T_u$ such that if $e$ is an edge of $S$ incident on $u$, then $G_e$ is elliptic in $T_u$,
\item for each edge $e=uv$ of $S$, points $x_{e,u}\in T_u$ and $x_{e,v}\in T_v$, both stabilized by $G_e$,
\item for each edge $e=uv$ of $S$, a segment $I_e=[y_u,y_v]$ (possibly of length $0$),
\end{itemize}
\noindent where all these data are $F_N$-equivariant, such that $T$ is obtained from the disjoint union of the trees $T_u$ and the segments $I_e$ by attaching every vertex $x_{e,u}\in T_u$ to the extremity $y_u$ of $I_e$. See e.g.\ \cite{Lev94}, although in the present paper, we allow some of the segments $I_e$ to have length $0$.

By a result of Levitt \cite{Lev94}, every tree $T\in \overline{cv_N}$
splits uniquely as a graph of actions in such a way that vertices of
the splitting correspond to connected components of the closure of the
set of branch points in $T$, and edges correspond to maximal arcs
whose interior contains no branch point of $T$. This splitting will be
refered to as the \emph{Levitt decomposition} of $T$. All vertex
actions $G_v\actson T_v$ of this decomposition have dense $G_v$-orbits
(the group $G_v$ might be trivial, and the tree $T_v$ might be reduced
to a point). The underlying simplicial tree $S$ of the splitting is
very small, and all its edges $e$ yield segments $I_e$ of positive
length in $T$. Every very small $F_N$-tree with dense orbits has
trivial arc stabilizers, see e.g. \cite[Remark~1.9]{BF94} or
\cite[Proposition~I.10]{GL95}. Therefore, every tree
$T\in\overline{cv_N}$ has only finitely many orbits of maximal arcs
with nontrivial stabilizer. Moreover, if $1\neq Z<F_N$ is cyclic the
fixed point set $Fix(Z)\subset T$ is empty, or a point, or an arc, and
if in addition $1\neq Z'<Z$ then $Fix(Z')=Fix(Z)$.

\paragraph*{Characteristic sets of elements in a very small $F_N$-tree.}
The {\it characteristic set} $\text{Char}_T(g)$ of an element
$g\in F_N$ in an $F_N$-tree $T$ is its axis if $g$ is hyperbolic and
its fixed point set if $g$ is elliptic. When $T$ is very small, the
characteristic set of a nontrivial elliptic element is a closed
interval (possibly a point). An important observation for us  is that if characteristic sets of $g$
and $h$ intersect in more than a point in $T$, then the same is true in a
neighborhood of (the homothety class of) $T$ in $\overline{CV_N}$.

\paragraph*{Morphisms between $F_N$-trees and a semi-flow on $\overline{CV_N}$.} A \emph{morphism} between two $F_N$-trees $T$ and $T'$ is an
$F_N$-equivariant map $f:T\to T'$, such that every segment $I\subseteq
T$ can be subdivided into finitely many subsegments $I_1,\dots,I_k$,
so that for all $i\in\{1,\dots,k\}$, the map $f$ is an isometry when 
restricted to $I_i$. Notice in particular that every morphism is $1$-Lipschitz. A morphism $f:T\to T'$ is \emph{optimal} if in
addition, for every $x\in T$, there is an open arc $I\subseteq T$
containing $x$ in its interior on which $f$ is one-to-one. We denote by $\mathcal{A}$ the space of isometry
classes of all $F_N$-trees, and by $\text{Opt}(\mathcal{A})$ the space of optimal morphisms between trees in $\mathcal{A}$, which is equipped with the equivariant Gromov--Hausdorff
topology, see \cite[Section 3.2]{GL07}. The following statement can be found in \cite[Section 3]{GL07}, it is based on work of Skora \cite{Sko89} inspired by an idea of Steiner.

\begin{prop}(Skora \cite{Sko89}, Guirardel--Levitt \cite{GL07})\label{factor}
There exist continuous maps $H:\text{Opt}(\mathcal{A})\times [0,1]\to \mathcal{A}$ and $\Phi,\Psi:\text{Opt}(\mathcal{A})\times [0,1]\to\text{Opt}(\mathcal{A})$ such that
\begin{itemize}
\item for all $f\in\text{Opt}(\mathcal{A})$ and all $t\in [0,1]$, the tree $H(f,t)$ is the range of the morphism $\Phi(f,t)$ and the source of the morphism $\Psi(f,t)$,
\item for all $f\in\text{Opt}(\mathcal{A})$, we have $\Phi(f,0)=id$ and $\Psi(f,0)=f$, 
\item for all $f\in\text{Opt}(\mathcal{A})$, we have $\Phi(f,1)=f$ and $\Psi(f,1)=id$, and 
\item for all $f\in\text{Opt}(\mathcal{A})$ and all $t\in [0,1]$, we have $\Psi(f,t)\circ\Phi(f,t)=f$.
\end{itemize}
\end{prop}

$$\xymatrix@C=1cm@R=1cm{
\text{source}(f)\ar[rr]^{f}\ar[dr]_{\Phi(f,t)}&&\text{range}(f)\\
& H(f,t)\ar[ur]_{\Psi(f,t)}&
}
$$

The proof of Proposition \ref{factor} goes as follows: given a morphism
$f:T_0\to T_1$, one first defines for all $t\in [0,1]$ a minimal
$F_N$-tree $T_t$, as the quotient space $T_0/{\sim_t}$, where $a\sim_t
b$ whenever $f(a)=f(b)$ and $\tau(a,b):=\sup_{x\in
  [a,b]}d_{T_1}(f(a),f(x))\le t$. The
morphism $f$ factors through optimal morphisms $\phi_t:T_0\to T_t$ and
$\psi_t:T_t\to T_1$ which vary continuously with $f$. 
We then let $H(f,t):=T_t$, $\Phi(f,t):=\phi_t$ and $\Psi(f,t):=\psi_t$. The path $(H(f,t))_{t\in [0,1]}$ will be called the \emph{canonical folding path} directed by $f$.

We will now make a few observations about the above construction. We recall that the \emph{bounded backtracking (BBT) constant} of a morphism $f:T_0\to T_1$, denoted by $BBT(f)$, is defined as the maximal real number such that for all $x,y\in T_0$ and all $z\in [x,y]$, we have $d_{T_1}(f(z),[f(x),f(y)])\le C$. We make the following observation.

\begin{lemma}\label{BBT}
For all $t\in [0,1]$, we have $BBT(\psi_t)\le BBT(f)$.
\end{lemma}

\begin{proof}
Let $x,y\in T_t$, and let $z\in [x,y]$. Let $x_0,y_0$ be $\phi_t$-preimages of $x,y$ in $T_0$. Then $\phi_t([x_0,y_0])$ contains $[x,y]$, so we can find a $\phi_t$-preimage $z_0$ of $z$ in the segment $[x_0,y_0]$. We then have $d_{T_1}(\psi_t(z),[\psi_t(x),\psi_t(y)])=d_{T_1}(f(z_0),[f(x_0),f(y_0)])\le BBT(f)$, which shows that $BBT(\psi_t)\le BBT(f)$.
\end{proof}

\begin{lemma}\label{id-time}
Assume that $f$ is isometric when restricted to any arc of $T_0$ with nontrivial stabilizer. Then for every $t\in [0,1]$, the map $\psi_t:T_t\to T_1$ is isometric when restricted to any arc of $T_t$ with nontrivial stabilizer.
\end{lemma}

\begin{proof}
Let $[a_t,b_t]\subseteq T_t$ be a nondegenerate arc with nontrivial stabilizer $\langle g\rangle$. We aim to show that $\psi_t(a_t)\neq \psi_t(b_t)$, which is enough to conclude since $\psi_t$ is a morphism. By definition of $T_t$, there exist $a,b\in T_0$ satisfying $\tau(a,ga)\le t$ and
$\tau(b,gb)\le t$, with $f(ga)=f(a)$ and $f(gb)=f(b)$, such that $a_t=\phi_t(a)$ and $b_t=\phi_t(b)$.

Assume towards a contradiction that $\psi_t(a_t)=\psi_t(b_t)$. Then $f(a)=f(b)$. If $g$ does not fix any nondegenerate arc in $T_0$, then the segment $[a,b]$ is contained in the union of all $g^k$-translates of $[a,ga]$ and $[b,gb]$, with $k$ varying over $\mathbb{Z}$. It follows that $\tau(a,b)\le\max\{\tau(a,ga),\tau(b,gb)\}\le t$, and hence $a_t=b_t$, a contradiction. Assume now that $g$ fixes a nondegenerate arc $[a',b']\subseteq T_0$, and let $a''$ (resp. $b''$) be the projection of $a$ (resp. $b$) to $[a',b']$. Using the fact that $f(a)=f(b)$ and that $f$ is isometric when restricted to $[a',b']$, we have $f([a'',b''])\subseteq f([a,a''])\cup f([b,b''])$, and therefore we get that $\tau(a,b)=\max\{\sup_{x\in [a,a'']}d_{T_1}(f(a),f(x)),\sup_{y\in [b,b'']}d_{T_1}(f(b),f(y))\}$. Since $[a,a'']\subseteq [a,ga]$ and $[b,b'']\subseteq [b,gb]$, we then obtain as above that $\tau(a,b)\le t$, so again $a_t=b_t$, a contradiction. 
\end{proof}

\begin{rk}\label{rk}
Together with \cite[Proposition 4.4]{Hor14-3}, which says that arc stabilizers in the intermediate trees are root-closed if arc stabilizers are root-closed in $T_0$ and $T_1$, Lemma~\ref{id-time} implies that if $T_0,T_1\in\overline{cv_N}$, and if $f$ is isometric when restricted to arcs with nontrivial stabilizer, then all intermediate trees belong to $\overline{cv_N}$. It is also known \cite[Proposition~3.6]{GL07} that if $T_0,T_1\in cv_N$, then all intermediate trees belong to $cv_N$.
\end{rk}

\subsection{Local contractibility at trees with trivial arc stabilizers}\label{sec-triv}

The goal of the present section is to prove that $\overline{cv_N}$ is locally contractible at every tree with all arc stabilizers trivial (Proposition~\ref{triv}). 
The following lemma provides nice approximations of trees in $\overline{cv_N}$ with trivial arc stabilizers.

\begin{lemma}(\cite[Theorem 5.3]{Hor14-2})\label{Lipschitz-approx}
Given a tree $T\in\overline{cv_N}$ with all arc stabilizers trivial, and any open neighborhood $\calu$ of $T$ in $\overline{cv_N}$, there exists a tree $U\in\calu\cap cv_N$ that admits an optimal morphism onto $T$.
\end{lemma}

\begin{proof}
It follows from
\cite[Theorem~3.6]{Hor14-1} that we can always find a simplicial tree
$U'\in\calu\cap\overline{cv_N}$ that admits an optimal morphism onto
$T$ (optimality is not stated in \cite{Hor14-1}, however it follows
from the construction which essentially relies on the approximation
techniques from \cite{Gui98}, see also Proposition~\ref{approx} for an
argument).
 
However, the $F_N$-action on $U'$ may \emph{a priori} not
be free (but $U'$ has trivial arc stabilizers because $T$ has trivial arc stabilizers). One way to replace $U'$ by a tree in $cv_N$ is to replace any vertex $v$ with nontrivial stabilizer $G_v$ by a free $G_v$-tree wedged at the point $v$, and
perform this operation equivariantly and at all nontrivial fixed vertices; however, the
natural map $U'\to U$ will collapse these trees and will not be a
morphism. Instead, we define a tree $U'_\epsilon$ in the following way (see Figure~\ref{fig-lipap}). 

Let $e_1,\dots,e_k$ be a choice of representatives of the orbits of edges incident on $v$ in $U'$, made such that for all $i\neq j$, the turn $\{e_i,e_j\}$ is $f$-legal, i.e.\ no initial segment of $e_i$ is identified with an initial segment of $e_j$ by $f$: this is possible because the tree $T$ has trivial arc stabilizers, so two edges in the same $G_v$-orbit never get identified by $f$. For all $i\in\{1,\dots,k\}$, we denote by $\ell(e_i)$ the length of the edge $e_i$. Let $\{a_1,\dots,a_l\}$ be a free basis of $G_v$. 

Let $U'_\epsilon$ be the tree obtained from $U'$ by giving length $\ell(e_1)-\epsilon$ to $e_1$, and blowing up the vertex $v$ to a free $G_v$-tree, as depicted in Figure~\ref{fig-lipap} (where we have represented the quotient graph $U'_\epsilon/F_N$). Then there is an optimal morphism $g:U'_\epsilon\to U'$, which is an isometry in restriction to each complementary component of the blown-up edges. We need to show that the composition $g\circ f$ is again optimal. Optimality at any point $x$ distinct from $w$ (with the notation from the picture) follows from the optimality of $g$. Optimality at $w$ follows from the fact that $a_1$ does not fix any nondegenerate arc in $T$. 
\end{proof}

\begin{figure}
\begin{center}
\input{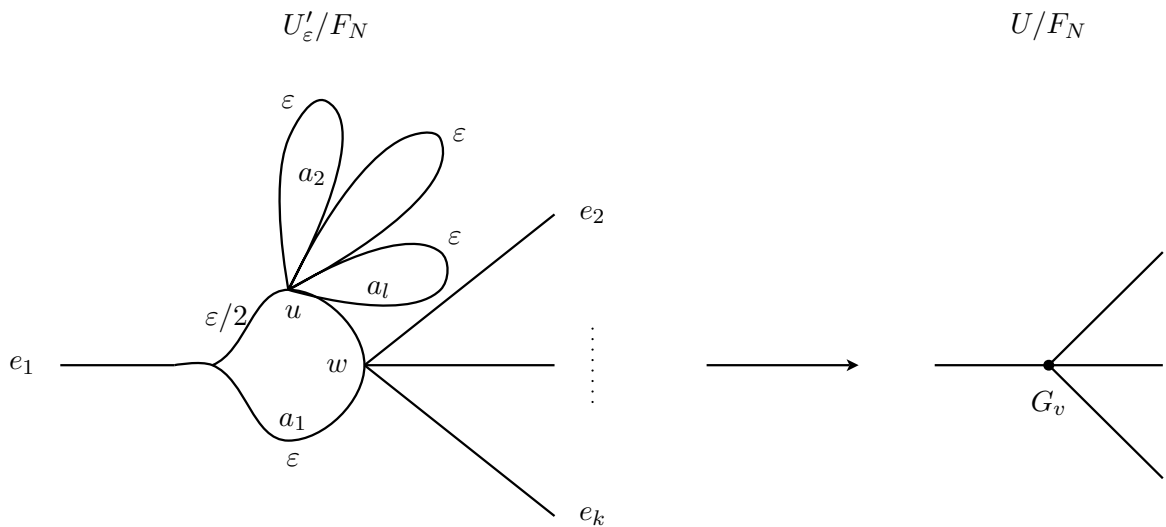}
\caption{The situation in the proof of Lemma~\ref{Lipschitz-approx}.}
\label{fig-lipap}
\end{center}
\end{figure}

\begin{lemma}\label{optimal-map}
Let $f:S\to T$ be an optimal morphism from a tree $S\in cv_N$ to a
tree $T\in \overline{cv_N}$. Then there is a neighborhood $\calw$ of
$T$ and a continuous
map $$\Psi_S:\calw\to\text{Opt}(\overline{cv_N})$$ such that for all
$W\in\calw$, the source of $\Psi_S(W)$ is a tree $S'\in cv_N$, in the
same (cone on a) simplex as $S$ and varying continuously, the range of
$\Psi_S(W)$ is the tree $W$, and $\Psi_S(W)$ is an optimal morphism.
In addition, $\Psi_S(T)=f$.
\end{lemma}

\begin{proof}
Let $v$ be a vertex of $S$. Since $f$ is optimal, the point $v$
belongs to a line $l\subseteq S$ such that the restriction $f_{|l}$ is
an isometry (notice however that we may not assume in general that $l$
is the axis of an element of $F_N$). Denote by $v_1$ and $v_2$ the two
vertices of $S$ which are adjacent to $v$ on the line $l$. We can then
find two hyperbolic elements $\gamma_{1,v},\gamma_{2,v}\in F_N$ whose
axes in $T$ both intersect $f(l)$ but do not intersect each other, and
such that the segment joining $Ax_T(\gamma_{1,v})$ to
$Ax_T(\gamma_{2,v})$ contains $f(v_1),f(v)$ and $f(v_2)$ in its
interior. Let $d\in\R$ denote the distance from $Ax_T(\gamma_{1,v})$
to $f(v)$.

If $W$ is sufficiently close to $T$, the elements $\gamma_{1,v}$ and
$\gamma_{2,v}$ are hyperbolic in $W$ and their axes are disjoint and lie at distance at least $d$ from each
other. We denote by $x^v_W$ the point at distance $d$ from
$Ax_W(\gamma_{1,v})$ on the segment from $Ax_W(\gamma_{1,v})$ to
$Ax_W(\gamma_{2,v})$. Given a choice $v_1,\dots,v_k$ of
representatives of the orbits of the vertices of $S$, there is a
unique choice of a (new) metric on $S$, giving a tree $S_W$, so that
the linear extension of $g_W$ (defined on vertices by sending
$v_i$ to $x^{v_i}_W$, and extending equivariantly) is a morphism
(up to restricting to a smaller neighborhood of $T$, we can assume that no edge gets length 0). Using the fact that
$f_{|l}$ is an isometry, we get that this morphism is also
optimal: indeed, the segment joining $Ax_T(\gamma_{1,v})$ to
$Ax_T(\gamma_{2,v})$ contains $[f(v_1),f(v_2)]$. It follows that when
$W$ is close to $T$ the segment joining $Ax_W(\gamma_{1,v})$ to
$Ax_W(\gamma_{2,v})$ overlaps $[g_W(v_1),g_W(v_2)]$ in a segment that
contains $g_W(v)$ in its interior. In particular, $g_W$ sends the two
directions at $v$ determined by $l$ to distinct directions.
Thus we set $\Psi_S(W)=g_W$. It is standard
that $\Psi_S$ is continuous, see \cite{GL07-2} for example.
\end{proof}

Given $T_0\in\overline{cv_N}$, the following corollary enables us to choose
basepoints continuously in all trees in a neighborhood of $T_0$, with
a prescribed choice on $T_0$. In the statement, we fix a Cayley tree
$R$ of $F_N$ with respect to a free basis of $F_N$, and a vertex
$\ast\in R$, and we denote by $\text{Map}(F_N,\overline{cv_N})$ the
collection of all $F_N$-equivariant maps from a tree obtained from $R$
by possibly varying edge lengths, to trees in $\overline{cv_N}$. 

\begin{cor}\label{basepoint}
  \begin{enumerate}[(i)]
    \item Let $c,c'\in F_N$ be two nontrivial elements that do not
      belong to the same cyclic subgroup. Then the
      function
      $$b:\overline{cv_N}\to \text{Map}(F_N,\overline{cv_N})$$ that
      sends $T\in \overline{cv_N}$ to the morphism $R\to T$ that sends
      the basepoint in $R$ to the projection of $Char_T(c)$ to
      $Char_T(c')$ when the two are disjoint, and to the midpoint of
      the overlap when they intersect, is continuous.
      \item 
Let $N\ge 2$, let $T_0\in\widehat{cv_N}$, let $A\subseteq F_N$ be a free factor of $F_N$, and let
$x_0\in T_0$ be a point which is contained in the union of all characteristic sets of elements of $A$.
\\ There exists a continuous map
$b:\overline{cv_N}\to\text{Map}(F_N,\overline{cv_N})$ such that for all
$T\in\overline{cv_N}$, the range of $b(T)$ is $T$, and  
$b(T_0)(\ast)=x_0$, and $b(T)(\ast)$ is contained in the union of all characteristic sets of elements of $A$.
\end{enumerate}
\end{cor}

\begin{proof} The first part of the corollary was proved in \cite[page 166]{GL07-2}. We prove (ii). From (i), we get a global choice of basepoints
  $b_1$ contained in the union of all characteristic sets of elements of $A$. From Lemma~\ref{optimal-map}, we have a choice of basepoints $b'_2$ defined in a
  neighborhood $\mathcal W$ of $T_0$, with $b'_2(T_0)(\ast)=x_0$. Let $a\in A$ be an element whose characteristic set contains $x_0$. By projecting $b'_2$ to the characteristic set of $a$, we get a continuous choice of basepoints $b_2$ in $\mathcal{W}$, all contained in the union of all characteristic sets of elements of $A$. Choose a continuous function
  $\phi:\overline{cv_N}\to [0,1]$ which is 1 at $T_0$ and is 0 outside
  a compact subset of $\mathcal W$. Then define
  $$b(T)(\ast)=(1-\phi(T))~ b_1(T)(\ast)+\phi(T)~ b_2(T)(\ast).$$
\end{proof}

\begin{lemma}\label{fold}
Let $T\in\overline{cv_N}$, let $\calu$ be an open neighborhood of $T$
in $\overline{cv_N}$. Then there exist $\epsilon>0$ and an open
neighborhood $\calv\subseteq\calu$ of $T$ in $\overline{cv_N}$ such
that if $U\in\calv$ is a tree that admits a $(1+\epsilon)$-Lipschitz
$F_N$-equivariant map $f$ onto $T$, and if $U'\in\overline{cv_N}$ is a
tree such that $f$ factors through $(1+\epsilon)$-Lipschitz
$F_N$-equivariant maps from $U$ to $U'$ and from $U'$ to $T$, then
$U'\in\calu$.
\end{lemma}

\begin{proof}
By definition of the equivariant Gromov--Hausdorff topology, there exist $\delta\in (0,1)$ and a finite set $\{g_1,\dots,g_k\}$ of elements of $F_N$ such that $\calu$ contains $$\calu':=\{T'\in\overline{cv_N}|(1-\delta)||g_i||_T\le ||g_i||_{T'}\le (1+\delta)||g_i||_T\text{~for all~}i\in\{1,\dots,k\}\}.$$ Let $\epsilon>0$ be such that $\epsilon<\min\{\delta,\frac{1}{1-\delta}-1\}$, and let $\delta'>0$ be such that $\delta'<\frac{1+\delta}{1+\epsilon}-1$ (this exists because $\epsilon<\delta$). Notice in particular that we have 
\begin{equation}\label{eq1}
1-\delta<\frac{1}{1+\epsilon}
\end{equation}
and 
\begin{equation}\label{eq2}
(1+\epsilon)(1+\delta')<1+\delta.
\end{equation}
Let $\calv$ be an open neighborhood of $T$ in $\overline{cv_N}$ contained in $$\{U\in\overline{cv_N}|(1-\delta')||g_i||_T\le ||g_i||_{U}\le (1+\delta')||g_i||_T\text{~for all~}i\in\{1,\dots,k\}\}.$$ If $U\in\calv$ and $U'\in\overline{cv_N}$ are trees such that $f$ factors through $(1+\epsilon)$-Lipschitz $F_N$-equivariant maps from $U$ to $U'$ and from $U'$ to $T$, then for all $i\in\{1,\dots,k\}$, we have $$||g_i||_T\le (1+\epsilon)||g_i||_{U'}\le (1+\epsilon)^2||g_i||_U\le (1+\epsilon)^2(1+\delta')||g_i||_T.$$ Using Equations~\eqref{eq1} and~\eqref{eq2}, this implies that $$(1-\delta)||g_i||_T\le ||g_i||_{U'}\le (1+\delta)||g_i||_T,$$ so $U'\in\calu$.  
\end{proof}

\begin{prop}\label{triv}
The space $\overline{cv_N}$ is locally contractible at every tree with all arc stabilizers trivial.
\end{prop}

\begin{proof}
Let $T\in\overline{cv_N}$ be a tree with all arc stabilizers trivial,
and let $\calu$ be an open neighborhood of $T$ in
$\overline{cv_N}$. Let $\epsilon>0$ and $\calv\subseteq\calu$ be a smaller neighborhood of $T$, as provided by Lemma \ref{fold}. Let $S\in \calv\cap cv_N$ be such that there exists
an optimal morphism from $S$ to $T$ (this exists by
Lemma~\ref{Lipschitz-approx}). Then there exists a smaller neighborhood $\calw\subseteq\calv$ of $T$ such that for all
$T'\in\calw$, there is a $(1+\epsilon)$-Lipschitz $F_N$-equivariant
map from $S$ to the source $S'$ of the morphism $\Psi_{S}(T')$ given
by Lemma \ref{optimal-map}. Since morphisms are $1$-Lipschitz, in view of Lemma \ref{fold}, this implies that all
trees that belong to either the straight path from $S$ to $S'$, or to
the canonical folding path directed by $\Psi_S(T')$, belong to
$\calu$. As $\Psi_S(T')$ varies continuously with $T'$, this gives a
homotopy of $\calw$ onto $S$ that stays within $\calu$.
\end{proof} 

A variant of the above argument shows the following statement, which will be useful in our proof of the fact that $\overline{CV_N}$ is not an AR.

\begin{lemma}\label{homotopy-lam}
Let $T_0\in\overline{cv_N}$ be a tree with trivial arc stabilizers that is dual to a measured foliation on a surface $\Sigma$ with a single boundary component $c$. Let $Z\subseteq\overline{cv_N}$ be the set of all trees dual to a measured foliation on $\Sigma$. \\ Then for every open neighborhood $\calv$ of $T_0$ in $\overline{cv_N}$, there exist a smaller neighborhood $\calw\subseteq\calv$ of $T_0$ in $\overline{cv_N}$, a tree $S\in\calw\cap cv_N$, and a continuous map $H:(Z\cap\calw)\times [0,1]\to\calv$ such that $H(z,0)=z$ and $H(z,1)=S$ for all $z\in Z\cap\calw$, and $c$ is hyperbolic in $H(z,t)$ for all $z\in Z$ and all $t>0$.
\end{lemma}

\begin{proof}
The proof of Lemma \ref{homotopy-lam} is the same as the proof of Proposition \ref{triv}, except that we have to show in addition that $c$ remains hyperbolic until it reaches $z$ along the canonical folding path from $S'$ to $z$ determined by the morphism $\Psi_S(z)$. Notice that $c$ is not contained in any proper free factor of $F_N$, so whenever $c$ becomes elliptic along a canonical folding path, the tree $T$ reached by the path contains no simplicial edge with trivial stabilizer. In view of Lemma \ref{id-time}, no two simplicial edges with nontrivial stabilizer can get identified by the folding process. In addition, an arc in a subtree with dense orbits from the Levitt decomposition of $T$ as a graph of actions cannot get identified with a simplicial edge with nontrivial stabilizer, and two such arcs cannot either get identified together \cite[Lemmas 1.9 and 1.10]{Hor14-1}. This implies that the canonical folding path $(H(\Psi_S(z),t))_{t\in [0,1]}$ becomes constant once it reaches $T$. To conclude the proof of Lemma \ref{homotopy-lam}, it remains to reparametrize this canonical folding path to ensure that it does not reach $T$ before $t=1$. Notice that $||c||_{H(\Psi_S(z),t)}$ decreases strictly as $t$ increases, until it becomes equal to $0$, and in addition the tree $H(\Psi_S(z),t)$ is the same for all $t\in [0,1]$ such that $||c||_{H(\Psi_S(z),t)}=0$. We can therefore reparametrize the canonical folding path by the translation length of $c$: for all $l\le ||c||_{H(\Psi_S(z),0)}$, we let $H'(\Psi_S(z),||c||_{H(\Psi_S(z),0)}-l)$ be the unique tree $T$ on the folding path for which $||c||_T=l$. To get a continuous map from $(Z\cap\calw)\times [0,1]$ to $\calv$, we then renormalize the parameter $l$ by dividing it by $||c||_{H(\Psi_S(z),0)}$.
\end{proof}

\subsection{The space $\overline{CV_N}$ is not locally $4$-connected when $N\ge 4$.}

We will now prove that Culler--Morgan's compactification $\overline{CV_N}$ of outer space is not an AR as soon as $N\ge 4$.

\begin{theo}\label{not-ar}
For all $N\ge 4$, the space $\overline{CV_N}$ is not locally $4$-connected, hence it is not an AR.
\end{theo}

\begin{rk}
It can actually be shown however that the closure $\overline{CV_2}$ is an absolute retract, and we also believe that $\overline{CV_3}$ is an absolute retract, though establishing this fact certainly requires a bit more work than the arguments from the present paper.
\end{rk}

\paragraph*{An embedded $3$-sphere in $\mathcal{PMF}^{vs}(\Sigma)$.}
 When $\Sigma$ is a compact surface of negative Euler characteristic $\chi(\Sigma)<0$ we denote by
$\mathcal{PML}(\Sigma)$ the space of projectivized measured geodesic laminations
on $\Sigma$ so that every boundary component has measure 0. 
The space $\mathcal{PML}(\Sigma)$ is homeomorphic to the sphere
$S^{-3\chi(\Sigma)-b-1}$ where $b$ is the number of boundary
components (see \cite[Theorem~3]{thurston} or \cite[Proposition~1.5]{hatcher}). 

We now specialize and let
$\Sigma$ be a nonorientable surface of genus $3$ with one boundary
component (so that its Euler characteristic is $-2$). Thus
$\mathcal{PML}(\Sigma)=S^4$. We mention that $\Sigma$ admits pseudo-Anosov homeomorphisms \cite{thurston,penner}.  We denote by
$\mathcal{PML}^{vs}(\Sigma)$ the subspace of $\mathcal{PML}(\Sigma)$
made of all laminations that are dual to very small trees
(equivalently, all laminations that do not contain any
$1$-sided compact leaf, see \cite{skora}). Let $\gamma$ be a simple
closed curve on $\Sigma$ that separates $\Sigma$ into a Möbius band
and an orientable surface $\Sigma_0$ (which is a compact surface of
genus $1$ with two boundary components). Denote by $\gamma_0$ the core of the Möbius band, which is $1$-sided and geodesic. Then the space
$\mathcal{PML}(\Sigma_0)=S^3$ is a subset of
$\mathcal{PML}^{vs}(\Sigma)$, however
$\gamma_0\notin\mathcal{PML}^{vs}(\Sigma)$. The key observation for constructing $4$-spheres
in $\overline{CV_N}$ showing that $\overline{CV_N}$ is not locally
$4$-connected will be the following.

\begin{lemma}\label{pmf}
The subset
$\mathcal{PML}(\Sigma_0)\subseteq\mathcal{PML}^{vs}(\Sigma)$ is a
$3$-sphere which is a retract of $\mathcal{PML}^{vs}(\Sigma)$.
\end{lemma}

\begin{proof}
The space $\mathcal{PML}(\Sigma_0)$ is a topologically embedded $3$-dimensional sphere in the $4$-dimensional sphere $\mathcal{PML}(\Sigma)$, so it separates
$\mathcal{PML}(\Sigma)$. We prove below that both sides of
$\mathcal{PML}(\Sigma)\setminus\mathcal{PML}(\Sigma_0)$ contain
curves arising as the core of a Möbius band, for which the dual
tree is not very small. This implies that
$\mathcal{PML}(\Sigma_0)$ is a retract of
$\mathcal{PML}^{vs}(\Sigma)$. 

To prove the assertion first notice that one of the two complementary
components is a cone of
$\mathcal{PML}(\Sigma_0)$ over $\gamma_0$. If $\phi$ is a pseudo-Anosov homeomorphism of $\Sigma$, then $\phi(\gamma_0)$ is the core of a M\"obius band, and it belongs to the other complementary
component.
\end{proof}

\paragraph*{Definition of the tree $T_0$.} We will now define a tree $T_0\in\overline{CV_N}$ at which $\overline{CV_N}$ will fail to be locally $4$-connected.

Let $N\ge 4$. Let $T_0\in\overline{CV_N}$ be the tree defined
in the following way (see Figure \ref{fig-tree}). Let
$\calf_0$ be an arational measured foliation on $\Sigma$, obtained as
the attracting foliation of a pseudo-Anosov diffeomorphism $f$ of
$\Sigma$. Let $A$ be the fundamental group of $\Sigma$, which is free of
rank $3$, let $T_A$ be the very small $A$-tree dual to $\calf_0$, and
let $x_A\in T_A$ be the unique point fixed by a nontrivial element
$c_A$ corresponding to the boundary curve of $\Sigma$. 

Let $B=F_{N-2}$, which we write as a free product $B=B'\ast\langle c_B\rangle$ for some element $c_B\in B$. Let $T_B\in\overline{cv_{N-2}}$ be a tree that splits as a graph of actions over this free splitting of $B$, with vertex actions a free and simplicial action $T_{B'}\in cv_{N-3}$, and the trivial action of $\langle c_B\rangle$ on a point, where the edge with trivial stabilizer from the splitting is given length $0$. Let $x_B$ be the point fixed by $c_B$ in $T_B$. Notice that $x_B$ belongs to the $B'$-minimal subtree $T_{B'}$ of $T_B$.

Write $F_N=A\ast_{c_A=c_B}B$, and let
$T_0$ be the very small $F_N$-tree obtained as a graph of actions over
this amalgamated free product, with vertex actions $T_A$ and $T_B$ and
attaching points $x_A$ and $x_B$, where the simplicial edges with
nontrivial stabilizers coming from the splitting are assigned length
$1$. We let $c:=c_A=c_B\in F_N$.

\begin{figure}
\begin{center}
\input{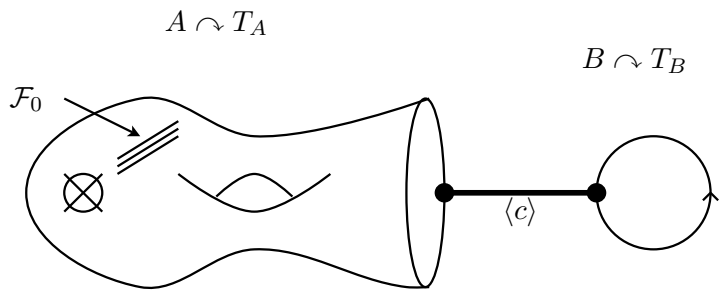}
\caption{The tree $T_0$ at which $\overline{CV_N}$ fails to be locally $4$-connected.}
\label{fig-tree}
\end{center}
\end{figure}

\paragraph*{Finding embedded $4$-spheres in a neighborhood of $T_0$.}
Choose an element $g\in F_N$ which is hyperbolic in $T_0$ and whose axis crosses the arc fixed by $c$. Let $\calu$ be an open neighborhood of $T_0$ 
in $\overline{CV_N}$ consisting of trees where $g$ is hyperbolic,
the characteristic sets of $g$ and $c$ overlap, and the surface
group $A$ is not elliptic. The crucial
property satisfied by such a neighborhood $\calu$ of $T_0$ is that we have the {\it oriented translation
  length} of $c$, namely the continuous function
$$\theta:\calu\to\R$$
defined by
$$\theta(T)=\epsilon_T \frac{||c||_T}{||g||_T}$$
where $\epsilon_T=1$ (resp. $-1$) if $c$ is hyperbolic in $T$ and the axes
of $c$ and $g$ give the same (resp. opposite) orientation to the
overlap, and $\epsilon_T=0$ if $c$ is elliptic. We will denote by $\calu_+$ (resp. $\calu_-$) the subset of $\calu$ made of trees such that $\theta(T)\ge 0$ (resp. $\theta(T)\le 0$). Similarly, given any smaller neighborhood $\calv\subseteq\calu$ of $T_0$, we will let $\calv_+:=\calv\cap\calu_+$ and $\calv_-:=\calv\cap\calu_-$.

We denote by $\overline{cv(A)}$ the space of all very small $A$-actions on $\mathbb{R}$-trees (which is also the closure of the outer space associated to $A$). Let $\calu_A$ be an open neighborhood of $T_A$ in $\overline{cv(A)}$ such that there exists an element $a\in A$ such that the characteristic sets of $c_A$ and $a$ have empty intersection in all trees in $\calu_A$.  For all $U_A\in\calu_A$, we then denote by $b(U_A)$ the projection of $\text{Char}_{U_A}(a)$ to $\text{Char}_{U_A}(c_A)$: this gives us a continuous choice of a basepoint in every tree $U_A\in\calu_A$. Given a tree $U_A\in\overline{cv(A)}$ in which $c_A$ is hyperbolic, and $t\in\mathbb{R}$, we let $b(U_A,t)$ be the unique point on the axis of $c_A$ at distance $|t|$ from $b(U_A)$, and such that $[b(U_A),b(U_A,t)]$ is oriented in the same direction as the axis of $c_A$ if $t\ge 0$, and in the opposite direction if $t\le 0$. We then let $\mathcal{T}(U_A,t)\in\overline{cv(A)}$ be the tree which splits as a graph of actions over the free product $F_N=A\ast B'$, with vertex actions $U_A$ and $T_{B'}$, and attaching points $b(U_A,t)$ and $x_B$ (recall from the above that $x_B\in T_{B'}$), where the simplicial edge from the splitting is assigned length $0$. The following lemma follows from the argument in the proof of \cite[Lemma 5.6]{Hor14-2}.

\begin{lemma}\label{lemma-sphere}
For every open neighborhood $\calv\subseteq\calu$ of $T_0$ in $\overline{CV_N}$, there exists an open neighborhood $\calv_A\subseteq\calu_A$ of $T_A$ in $\overline{cv(A)}$ such that for all $U_A\in\calv_A$ such that $c_A$ is hyperbolic in $U_A$, we have $\mathcal{T}(U_A,1)\in\calv_+$, and $\mathcal{T}(U_A,-1)\in\calv_-$.
\qed
\end{lemma}

Denote by $X(c)$ the subspace of $\overline{CV_N}$ made of all trees where $c$ is elliptic, and by $X(c)^*$ the subset of $X(c)$ consisting of trees where the
surface group $A$ is not elliptic. Notice that our choice of neighborhood $\calu$ of $T_0$ ensures that $\calu\cap X(c)\subseteq X(c)^{\ast}$. For any tree $T$ in $X(c)^\ast$, the minimal $A$-invariant subtree is dual to some measured foliation on
$\Sigma$ \cite{skora}, so we have a map $\psi:X(c)^\ast\to
\mathcal{PML}^{vs}(\Sigma)$.

\begin{prop}\label{bad-sphere}
For every open neighborhood
$\mathcal{V}\subseteq\calu$ of $T_0$, there exists a topologically
embedded $3$-sphere $S^3$ in $X(c)\cap\calv$ which is a retract of $X(c)\cap\calu$, and such that the inclusion map $\iota:S^3\hookrightarrow X(c)\cap\calv$ extends to continuous maps $\iota_{\pm}:B^4_{\pm}\to\mathcal{V}_{\pm}$ (where $B^4_{\pm}$ are $4$-balls with boundary $S^3$), with $\iota_{\pm}(B^4_{\pm}\setminus S^3)\subseteq \overline{CV_N}\setminus X(c)$.
\end{prop}

\begin{proof}
We identify $\mathcal{PML}(\Sigma)$ with a continuous lift in
$\mathcal{ML}(\Sigma)$. This gives a continuous injective map
$\phi:\mathcal{PML}^{vs}(\Sigma)\to X(c)^\ast$, mapping every measured
foliation $\mathcal{F}$ to the tree in $\overline{CV_N}$ that splits
as a graph of actions over $F_N=A\ast_{\langle c\rangle} B$, with vertex actions the
tree dual to $\mathcal{F}$, and the fixed action $B\actson T_B$, where the attaching points are the unique points fixed by $c$, and the simplicial edge from the splitting is assigned length $1$. The
composition $\psi\phi$ is the identity. Let $\Sigma_0\subset\Sigma$ be an oriented subsurface as in Lemma~\ref{pmf}. We can assume that the $3$-sphere $\phi(\mathcal{PML}(\Sigma_0))$ is contained in $\calv$: to achieve this, use uniform north-south dynamics of the pseudo-Anosov
homeomorphism $f$ on $\mathcal{PML}(\Sigma)$ to find
$k\in\mathbb{Z}$ so that $\phi(\mathcal{PML}(f^k(\Sigma_0)))\subseteq
X(c)\cap\calv$ (see \cite[Theorem 3.5]{Iva92}, which is stated there in the case of an orientable surface, but the similar statement for a nonorientable surface follows by considering the orientable double cover of $\Sigma$). Using Lemma~\ref{pmf}, we get that the 3-sphere
$\phi(\mathcal{PML}(\Sigma_0))$ is a retract of $X(c)\cap\calu$ (notice also that its image in $\overline{cv(A)}$ is again an embbeded $3$-sphere which we denote by $S_A^3$).  

Let now $\calv_A$ be an open neighborhood of $T_A$ in $\overline{cv(A)}$ provided by Lemma \ref{lemma-sphere}, and let $\calw_A\subseteq\calv_A$ be a smaller neighborhood of $T_A$ provided by Lemma \ref{homotopy-lam}. We can assume that the sphere constructed in the above paragraph is such that $S_A^3\subseteq\calw_A$. By Lemma~\ref{homotopy-lam}, there exists a tree $V_A\in\calv_A\cap cv(A)$ and a homotopy $H_A:S_A^3\times [0,1]\to \calv_A$ such that 
\begin{itemize}
\item for all $z\in S_A^3$, we have $H_A(z,0)=z$, and $H_A(z,1)=V_A$, and
\item for all $z\in S_A^3$ and all $t>0$, the element $c$ is hyperbolic in $H_A(z,t)$.
\end{itemize}
\noindent We now define $H_{\pm}:S^3\times [0,1]\to \calv_{\pm}$ by letting $H_{\pm}(z,0)=z$ for all $z\in S^3$, and $H_{\pm}(z,t)=\mathcal{T}(H_A(z,t),\pm 1)$ for all $z\in S^3$ and all $t>0$. Continuity follows from the argument from \cite[Lemma 5.6]{Hor14-2}. This enables us to construct the maps $\iota_{\pm}$, where the segment joining a point $z\in S^3$ to the center of $B^4_{\pm}$, is mapped to $H_{\pm}(z,[0,1])$. 
\end{proof}

Our proof of Theorem \ref{not-ar} uses a \v{C}ech homology argument, given in Appendix \ref{sec-Cech} of the paper. 

\begin{proof}[Proof of Theorem \ref{not-ar}]
We will show that for every neighborhood $\calv\subset\calu$ of $T_0$ there is a map
$S^4\to \calv$ which is not nullhomotopic in $\calu$. This is illustrated in Figure \ref{fig-sphere}. Let $S^3$ be a topologically embedded $3$-sphere in $X(c)\cap\calv$
provided by Proposition~\ref{bad-sphere}. In particular, the
inclusion map $\iota:S^3\hookrightarrow X(c)\cap\calu$ is nontrivial
in homology (either singular homology, or \v{C}ech homology $\check{H}_3$ with $\mathbb{Z}/2$ coefficients). In addition, we can extend $\iota$ to maps
$\iota_{\pm}:B^4_\pm \to\mathcal{V}_{\pm}$, with
$\iota_{\pm}(B^4_{\pm}\setminus S^3)\subseteq \overline{CV_N}\setminus X(c)$.
Gluing these two along the boundary produces a map $f:S^4\to\calv$, which restricts to
$\iota$ on $S^3$, and such that the composition $h=\theta f:S^4\to\R$ (where we recall that $\theta$ is the oriented translation length of $c$) is
standard (i.e. $h^{-1}(0)=S^3$, $h^{-1}([0,+\infty))=B^4_+$ and $h^{-1}((-\infty,0])=B^4_-$).

Now suppose that $f$ extends to $\tilde f:B^5\to\calu$, and let $\tilde h=\theta \tilde f:B^5\to\R$. If $\tilde h^{-1}(0)$ were a manifold with boundary $S^3$, then we would immediately deduce that the inclusion $$S^3\hookrightarrow\tilde h^{-1}(0)$$ is trivial in singular homology $H_3$. This may fail to be true, but we can still apply Lemma~\ref{Cech} to deduce that the above inclusion is trivial in \v{C}ech homology $\check{H}_3$ with $\mathbb{Z}/2$ coefficients. Applying $\tilde f$ and observing that
$\tilde f(\tilde h^{-1}(0))\subseteq X(c)\cap\calu$ we deduce that
$$f_{|S^3}=\iota:S^3\to X(c)\cap\calu$$ is trivial in $\check{H}_3$, a contradiction.
\end{proof}

\begin{figure}
\begin{center}
\def\JPicScale{.8}
\input{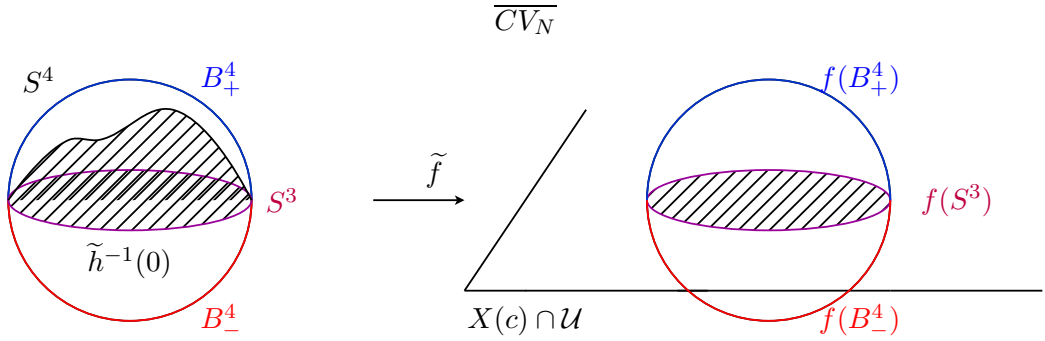}
\caption{Constructing the sphere $S^4$ showing failure of local $4$-connectivity of $\overline{CV_N}$.}
\label{fig-sphere}
\end{center}
\end{figure}

\section{The Pacman compactification of outer space}\label{sec-pacman}

In the present paper, we will be interested in another
compactification $\widehat{CV_N}$ of outer space, which we call the
\emph{Pacman compactification}. We also define $\widehat{cv_N}$ as the
unprojectivized version of $\widehat{CV_N}$. We will now define
$\widehat{CV_N}$ and establish some basic topological
properties.

A point in $\widehat{CV_N}$ is given by a (homothety class of a) tree $T\in\overline{CV_N}$,
together with an $F_N$-equivariant choice of orientation of the
characteristic set of every nontrivial element of $F_N$ that fixes a
nondegenerate arc in $T$ (notice that every element $g\in F_N$ which is hyperbolic in $T$ determines a natural orientation on its axis, which is isometric to a real line, and the element $g^{-1}$ determines the opposite orientation on this line). Precisely, given any nontrivial element $c\in F_N$
which is not a proper power, and whose characteristic set in $T$ is a
nondegenerate arc $\text{Char}_T(c):=[x,y]$ fixed by $c$, we
prescribe an orientation of $\text{Char}_T(c)$, and the orientation of
$\text{Char}_T(c^{-1})$ is required to be opposite to the orientation
of $\text{Char}_T(c)$. The $F_N$-translates of $[x,y]$ get the
induced orientation as required by $F_N$-equivariance. Notice that there is an
$F_N$-equivariant surjective map
$\pi:\widehat{CV_N}\to\overline{CV_N}$, which consists in forgetting
the orientations of the edges with nontrivial stabilizer. 

Given two oriented (possibly finite or infinite) geodesics $l$ and
$l'$ in an $\mathbb{R}$-tree $T$ with nondegenerate intersection, we
define the \emph{relative orientation} of $l$ and $l'$ as being equal
to $+1$ if the orientations of $l$ and $l'$ agree on their
intersection, and $-1$ otherwise. Given $T\in\widehat{CV_N}$, and two
elements $\alpha,\beta\in F_N$ whose characteristic sets have
nondegenerate intersection in $T$, we define the \emph{relative
  orientation} of the pair $(\alpha,\beta)$ in $T$ as being equal to
the relative orientation of their characteristic sets.

 We now define a topology on $\widehat{CV_N}$. Given an open set
$U\subseteq \overline{CV_N}$, and a finite (possibly empty) collection of pairs
$(\alpha_1,\beta_1),\dots,(\alpha_k,\beta_k)$ of elements of $F_N$,
such that for all $i\in\{1,\dots,k\}$, $\beta_i$ is hyperbolic and the
characteristic sets of $\alpha_i$ and $\beta_i$ have nondegenerate
intersection in all trees in $U$, we
let $$U((\alpha_1,\beta_1),\dots,(\alpha_k,\beta_k))$$ be the set of
all $T\in\widehat{CV_N}$ such that $\pi(T)\in U$, and the relative
orientation of $(\alpha_i,\beta_i)$ in $T$ is equal to $+1$ for all
$i\in\{1,\dots,k\}$. Notice that
given two open sets $U,U'\subseteq\overline{CV_N}$, and finite
collections of elements $\alpha_i,\beta_i,\alpha'_j,\beta'_j\in F_N$,
the intersection $U((\alpha_1,\beta_1),\dots,(\alpha_k,\beta_k))\cap
U'((\alpha'_1,\beta'_1),\dots,(\alpha'_l,\beta'_l))$ is equal to
$(U\cap
U')((\alpha_1,\beta_1),\dots,(\alpha_k,\beta_k),(\alpha'_1,\beta'_1),\dots,(\alpha'_l,\beta'_l))$. This
shows the following lemma.

\begin{lemma}
The sets $U((\alpha_1,\beta_1),\dots,(\alpha_k,\beta_k))$ form a basis of open sets for a topology on $\widehat{CV_N}$.
\qed
\end{lemma}

From now on, we will equip $\widehat{CV_N}$ with the topology generated by these sets. Since trees in $CV_N$ have trivial arc stabilizers, there is an inclusion map $\iota:CV_N\hookrightarrow \widehat{CV_N}$. 

\begin{lemma}\label{embed}
The map $\iota:CV_N\hookrightarrow \widehat{CV_N}$ is a topological embedding. 
\end{lemma}

\begin{proof}
It follows from the definition of the Gromov--Hausdorff topology that if $T\in CV_N$, and if $(\alpha_1,\beta_1),\dots,(\alpha_k,\beta_k)$ is a finite set of pairs of elements of $F_N$ such that for all $i\in\{1,\dots,k\}$, the axes of $\alpha_i$ and $\beta_i$ in $T$ have nondegenerate intersection, and their orientations agree on their intersection, then the same remains true in a neighborhood of $T$. Therefore, the topology on $\widehat{CV_N}$ restricts to the Gromov--Hausdorff topology on $CV_N$, which shows that $\iota$ is a topological embedding. 
\end{proof}

For every tree $T\in \overline{CV_N}$, there are only finitely many conjugacy classes of elements of $F_N$ that fix a nondegenerate arc in $T$. Therefore, the map $\pi:\widehat{CV_N}\to\overline{CV_N}$ has finite fibers. The map $\pi$ is continuous for the above topology on $\widehat{CV_N}$ because the $\pi$-preimage of any open set $U\subseteq\overline{CV_N}$ is the open set $U(\emptyset)$ in $\widehat{CV_N}$.

\begin{prop}
The space $\widehat{CV_N}$ is second countable.
\end{prop}

\begin{proof}
Since $\overline{CV_N}$ is second countable, we can choose a countable basis $(U_i)_{i\in\mathbb{N}}$ of open sets of $\overline{CV_N}$. Then the countable collection of all sets of the form $U_i((\alpha_1,\beta_1),\dots,(\alpha_k,\beta_k))$ (where for all $j\in\{1,\dots,k\}$, the element $\beta_j$ is hyperbolic and the characteristic sets of $\alpha_j$ and $\beta_j$ have nondegenerate intersection in all trees in $U_i$) is a basis of open sets of $\widehat{CV_N}$. 
\end{proof}

\begin{prop}
The space $\widehat{CV_N}$ is Hausdorff. 
\end{prop}

\begin{proof}
Let $T\neq T'\in\widehat{CV_N}$. If $\pi(T)\neq\pi(T')$, then since
$\overline{CV_N}$ is Hausdorff, we can find disjoint open
neighborhoods $U$ of $\pi(T)$ and $U'$ of $\pi(T')$ in $\overline{CV_N}$, and
these yield disjoint open neighborhoods of $T$ and $T'$ in
$\widehat{CV_N}$. So we can assume that $\pi(T)=\pi(T')$, and there is
an arc $e$ with nontrivial stabilizer $\langle c\rangle$ in $\pi(T)$ whose
orientation is not the same in $T$ and $T'$. Let $h\in F_N$ be an
element which is hyperbolic in $\pi(T)$, whose axis contains the edge $e$. One can then find an open neighborhood $U$
of $\pi(T)$ in $\overline{CV_N}$ such that for all trees $Y\in U$, the element $h$ is hyperbolic in $Y$, and the characteristic sets of $h$ and $c$ have nondegenerate intersection in $Y$. Then exactly one of the points $T,T'\in\widehat{CV_N}$ belongs to
$U(c,h)$, the other belongs to $U(c,h^{-1})$, and $U(c,h)\cap
U(c,h^{-1})=\emptyset$. This shows that $\widehat{CV_N}$ is Hausdorff.
\end{proof}

The following lemma gives a useful criterion for checking that a sequence converges in $\widehat{CV_N}$.

\begin{lemma}\label{proving-cv}
Let $(T_n)_{n\in\mathbb{N}}\in\widehat{CV_N}^{\mathbb{N}}$, and let $T\in\widehat{CV_N}$. Assume that $(\pi(T_n))_{n\in\mathbb{N}}$ converges to $\pi(T)$ in $\overline{CV_N}$, and that for every element $\alpha\in F_N$ fixing a nondegenerate arc in $T$, there exists a hyperbolic element $\beta$ whose axis in $T$ has nondegenerate intersection with $\text{Char}_T(\alpha)$, such that the relative orientations of $(\alpha,\beta)$ eventually agree in $T_n$ and in $T$. Then $(T_n)_{n\in\mathbb{N}}$ converges to $T$.
\end{lemma}

\begin{proof}
 As $(\pi(T_n))_{n\in\mathbb{N}}$ converges to $\pi(T)$, it is enough to prove that given any two elements $\alpha,\beta'\in F_N$ whose characteristic sets in $T$ have nondegenerate intersection, with $\beta'$ hyperbolic in $T$, the relative orientations of $(\alpha,\beta')$ eventually agree in $T_n$ and in $T$. 

This is clearly true if $\alpha$ is also hyperbolic in $T$, so we assume that $\alpha$ fixes a nondegenerate arc in $T$. By hypothesis, there exists $\beta\in F_N$ whose axis in $T$ has nondegenerate intersection with $\text{Char}_T(\alpha)$, such that the relative orientation of $(\alpha,\beta)$ is equal to $1$ both in $T$ and in $T_n$ for $n$ sufficiently large. 

\begin{figure}
\begin{center}
\input{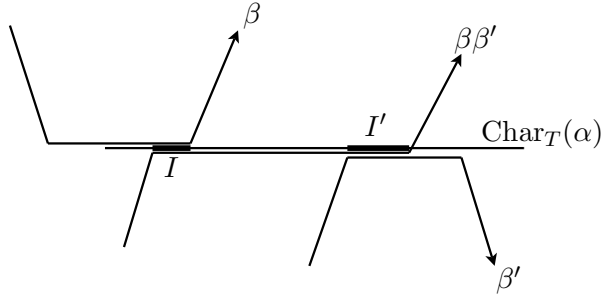}
\caption{Typical position of the characteristic set of $\alpha$ and the axes of $\beta,\beta'$ and $\beta\beta'$ in the proof of Lemma~\ref{proving-cv}.}
\label{fig-beta}
\end{center}
\end{figure}

Since $\beta$ and $\beta'$ are hyperbolic in $T$, they are also both hyperbolic in all trees $T_n$ with $n\in\mathbb{N}$ large enough. A typical situation is depicted in Figure~\ref{fig-beta}. Up to replacing $\beta$ and $\beta'$ by their inverses, we can assume that $\beta\beta'$ is hyperbolic in $T$, and that the characteristic sets of $\alpha$, $\beta$ and $\beta\beta'$ in $T$ contain a common nondegenerate segment $I$, and likewise the characteristic sets of $\alpha$, $\beta'$ and $\beta\beta'$ contain a common nondegenerate segment $I'$. This also holds in $T_n$ for all sufficiently large $n\in\mathbb{N}$ (in particular $\beta\beta'$ is hyperbolic in $T_n$ for all sufficiently large $n\in\mathbb{N}$). Since $\beta$ and $\beta\beta'$ are both hyperbolic, their relative orientation eventually agrees in $T_n$ and in $T$. By hypothesis, so does the relative orientation of $(\alpha,\beta)$. The existence of the segment $I$ thus ensures that the relative orientation of $(\alpha,\beta\beta')$ eventually agrees in $T_n$ and in $T$. Since $\beta'$ and $\beta\beta'$ are both hyperbolic, their relative orientation eventually agrees in $T_n$ and in $T$. The existence of the segment $I'$ thus ensures that the relative orientation of $(\alpha,\beta')$ eventually agrees in $T_n$ and in $T$, as required.
\end{proof}

\begin{prop}
The space $\widehat{CV_N}$ is compact.
\end{prop}

\begin{proof}
Since $\widehat{CV_N}$ is second countable, it is enough to prove
sequential compactness. Let
$(T_n)_{n\in\mathbb{N}}\in{\widehat{CV_N}}^{\mathbb{N}}$. Since
$\overline{CV_N}$ is compact \cite[Theorem 4.5]{CM87}, up to passing
to a subsequence, we can assume that $(\pi(T_n))_{n\in\mathbb{N}}$
converges to a tree $T\in\overline{CV_N}$. 

Let $\alpha\in F_N$ be an element that fixes a nondegenerate arc in $T$, and let $\beta\in F_N$ be a hyperbolic element in $T$, whose axis has nondegenerate intersection with $\text{Char}_T(\alpha)$. Then up to passing to a subsequence, we can assume that the relative orientation of $(\alpha,\beta)$ in $T_n$ is eventually constant, and assign the corresponding orientation to $\text{Char}_T(\alpha)$. If we do this equivariantly for each of the finitely many orbits of maximal arcs with nontrivial stabilizer in $T$, Lemma \ref{proving-cv} ensures that we have found a subsequence of $(T_n)_{n\in\mathbb{N}}$ that
converges to $T$. 
\end{proof}

Being compact, Hausdorff, and second countable, the space
$\widehat{CV_N}$ is metrizable.

\begin{cor}
The space $\widehat{CV_N}$ is metrizable.
\qed
\end{cor}

\begin{prop}
The space $\widehat{CV_N}$ has finite topological dimension equal to $3N-4$.
\end{prop}

\begin{proof}
The map $\pi:\widehat{CV_N}\to\overline{CV_N}$ is a continuous map
between compact metrizable spaces, with finite point preimages. It
follows from Hurewicz's theorem (see \cite[Theorem~4.3.6]{Eng78}) that
$\text{dim}(\widehat{CV_N})\le\text{dim}(\overline{CV_N})$. In
addition, the space $\widehat{CV_N}$ contains $CV_N$ as a
topologically embedded subspace (Lemma~\ref{embed}). As
$\text{dim}(\overline{CV_N})=\text{dim}(CV_N)=3N-4$ (see
\cite{BF94,GL95}), the result follows.
\end{proof}

Corollary~\ref{is-closure} below will show in addition that $CV_N$ is dense in $\widehat{CV_N}$, so $\widehat{CV_N}$ is a compactification of $CV_N$. The
$Out(F_N)$-action clearly extends to a continuous action on this compactification.

\section{The space $\widehat{CV_N}$ is an AR, and the boundary is a $Z$-set.}\label{sec-ar}

The goal of the present section is to prove Theorem 2 from the
introduction. 

\subsection{General strategy}

We first explain the general strategy of our proof. Let $X$ be a topological space, let $Z\subset X$ be a closed subspace, and let
$z\in Z$. We say that $(X,Z)$ is {\it locally complementarily
  contractible (LCC) at $z$} if for every neighborhood $\mathcal U$ of
$z$ in $X$ there is a smaller neighborhood $\mathcal V$ of $z$ in $X$
such that the inclusion $\mathcal V\smallsetminus Z\hookrightarrow
\mathcal U\smallsetminus Z$ is nullhomotopic. If this is true for
every $z\in Z$ we say that $(X,Z)$ is {\it LCC}. 

\begin{rk}\label{lcc-dense}
Notice that if $(X,Z)$ is LCC, then $X\smallsetminus Z$ is dense in $X$, as otherwise we would be able to find a point $z\in Z$ and an open neighborhood $\calv$ of $z$ in $X$ such that $\calv\smallsetminus Z=\emptyset$, and the map $\emptyset\to\emptyset$ is not nullhomotopic by convention. 
\end{rk}

To detect that
$\widehat{CV_N}$ is an AR we will use the criterion in the appendix,
Theorem~\ref{miracle}. We will apply it with $X=\widehat{CV_N}$ and
$Z=X\smallsetminus CV_N$. We are thus reduced to proving that $(X,Z)$
is LCC at every $T_0\in Z$. We will actually work in the unprojectivized version $\widehat{cv_N}$ and prove the following statement.

\begin{theo}\label{cv-lcc}
For all $N\ge 1$, the pair $(\widehat{cv_N},\widehat{cv_N}\smallsetminus cv_N)$ is LCC.
\end{theo}

In view of Remark~\ref{lcc-dense}, we deduce the following result.

\begin{cor}\label{is-closure}
For all $N\ge 1$, every point in $\widehat{cv_N}$ is a limit of points in $cv_N$.
\qed
\end{cor}

\begin{proof}[Proof of Theorem~2 from Theorem~\ref{cv-lcc}]
Since $CV_N$ is contractible and locally contractible, and since $\widehat{CV_N}$ is compact, metrizable and finite-dimensional, in view of Theorem~\ref{miracle}, it is enough to show that $(\widehat{CV_N},\widehat{CV_N}\smallsetminus CV_N)$ is LCC. 

Denote by $\ast$ the trivial tree in $\widehat{cv_N}$. We claim that there exists a continuous lift $\phi:\widehat{CV_N}\to\widehat{cv_N}$ of the projection map $\psi:\widehat{cv_N}\smallsetminus\{\ast\}\to\widehat{CV_N}$. Indeed, by \cite[I.6.5, Corollaire 2]{Ser}, there exists a finite set $\{g_1,\dots,g_k\}$ of elements of $F_N$ such  that for every nontrivial tree $T\in\widehat{cv_N}$, there exists $i\in\{1,\dots,k\}$ with $||g_i||_T>0$. The lift $\phi$ is then defined by sending a point $[T]\in\widehat{CV_N}$ to the unique representative for which $\sum_{i=1}^k||g_i||_T=1$.

 Let $[T]\in\widehat{CV_N}$, and let $T$ be a lift to $\widehat{cv_N}$. Let $\calu_{[T]}$ be a neighborhood of $[T]$ in $\widehat{CV_N}$, and let $\calu_T$ be the full preimage of $\calu_{[T]}$ in $\widehat{cv_N}$, which is a neighborhood of $T$ in $\widehat{cv_N}$. 
Since $(\widehat{cv_N},\widehat{cv_N}\smallsetminus cv_N)$ is LCC at $T$, there exists a neighborhood $\calv_T\subseteq\calu_T$ of $T$ in $\widehat{cv_N}$ such that the inclusion $\calv_T\cap cv_n\subseteq\calu_T\cap cv_N$ is nullhomotopic; we denote by $H:(\calv_T\cap cv_N)\times [0,1]\to\calu_T\cap cv_N$ a homotopy from $\calv_T\cap cv_N$ to a point. 
Let $\calv_{[T]}$ be a neighborhood of $[T]$ in $\widehat{CV_N}$ such that $\phi(\calv_{[T]})\subseteq\calv_T$. Then $\psi\circ H\circ\phi$ is a homotopy from $\calv_{[T]}\cap CV_N$ to a point that stays inside $\calu_{[T]}\cap CV_N$. This shows that $(\widehat{CV_N},\widehat{CV_N}\smallsetminus CV_N)$ is LCC, as claimed. 
\end{proof}

We now explain our proof of Theorem~\ref{cv-lcc}; Theorem~\ref{strategy} below summarizes the
strategy. In order to state it, we extend the notion of morphisms to $\widehat{cv_N}$ as follows: a \emph{morphism} between two
trees $T,T'\in\widehat{cv_N}$ is a morphism between $\pi(T)$ and
$\pi(T')$ which is further assumed to be isometric and
orientation-preserving when restricted to every edge with nontrivial
stabilizer.

\begin{theo}\label{strategy}
  Let $T_0\in\widehat{cv_N}$ be a nontrivial tree and let $\mathcal U$
  be a neighborhood of $T_0$. Then there is a tree $U\in\mathcal U$
  and an optimal morphism $f:U\to T_0$ such that
  \begin{enumerate}
    \item [(1)] $U$ splits as a graph of actions over a 1-edge
      free splitting $S$ of $F_N$.
  \end{enumerate}
  Further, there is a subset $X_{S,U}\subseteq\widehat{cv_N}$
  that
  contains $U$, a continuous map $\rho:\widehat{cv_N}\to X_{S,U}$, and a homotopy
  $H:\widehat{cv_N}\times [0,1]\to \widehat{cv_N}$ between the identity and $\rho$
  such that
  \begin{enumerate}
  \item [(2)] $\rho(T_0)=U$,
  \item [(3)] $H(\{T_0\}\times [0,1])\subseteq \mathcal U$,
  \item [(4)] $H(cv_N\times [0,1])\subseteq cv_N$, and
  \item [(5)] Assume that the pair $(\widehat{cv_k},\widehat{cv_k}\smallsetminus cv_k)$ is LCC for all $k<N$. If $U\not\in cv_N$
        then $(X_{S,U},X_{S,U}\smallsetminus {cv_N})$ is LCC at $U$.
  \end{enumerate}
\end{theo}

The fact that every tree in $\widehat{cv_N}$ can be approximated by
trees that split as graphs of actions over free splittings
(Property~(1)) will be explained in Section~\ref{sec-approx-split}; it
relies on classical approximation arguments, using the Rips
machine. The construction of $X_{S,U}$ and of the maps $\rho$ and $H$
will be carried out in Section~\ref{sec-construction}. Roughly
speaking, the set $X_{S,U}$ will consist of trees that split as a
graph of actions over the one-edge free splitting $S$, but the details
of the definition will depend on $U$. To prove Property~(5), we will
take advantage of this splitting, which
will allow for an inductive argument.

\begin{proof}[Proof of Theorem~\ref{cv-lcc} assuming Theorem \ref{strategy}]
  The proof is by induction on $N$. The statement is obvious when $N=1$, so we assume that $N\ge 2$. 
  
 First, we observe that the pair
 $(\widehat{cv_N},\widehat{cv_N}\smallsetminus cv_N)$ is LCC at the
 trivial tree $\ast$. Indeed, by \cite[I.6.5, Corollary~2]{Ser}, we can find a finite set $\{\gamma_1,\dots,\gamma_k\}$ of elements of $F_N$ such that for every nontrivial tree $T\in\widehat{cv_N}$, there exists $i\in\{1,\dots,k\}$ such that $||\gamma_i||_T>0$. Then any open
 set $\mathcal U$ that contains $*$ contains an open set of the form
 $$\mathcal V=\{T\in\widehat{cv_N}\mid \sum_{i=1}^k||\gamma_i||_T<\epsilon\}$$
 for some $\epsilon>0$. By scaling, $\mathcal V\cap cv_N$ deformation
 retracts to the slice
 $$\{T\in cv_N\mid \sum_{i=1}^k||\gamma_i||_T=\epsilon/2\}$$
 which is homeomorphic to $CV_N$ and hence contractible. This shows
 that $\mathcal V\cap cv_N\hookrightarrow \mathcal U\cap cv_N$ is
 nullhomotopic. 
  
  Let now $T_0\in \widehat{cv_N}$ be a nontrivial tree and let
  $\mathcal U$ be a neighborhood of $T_0$. Let $U, X_{S,U},\rho, H$ be
  as in Theorem \ref{strategy}. There are now two cases, depending whether $U$ belongs to $cv_N$ or to the boundary. 
  
  If $U\in cv_N$, choose a contractible neighborhood $\mathcal W$ of $U$ in
  $cv_N\cap\mathcal U$. By the continuity of $\rho$ and $H$ and properties (2)
  and (3), there is a
  neighborhood $\mathcal V\subset\mathcal U$ of $T_0$ such that
  $H(\mathcal V\times [0,1])\subseteq\mathcal U$ and
  $\rho(\mathcal V)\subseteq\mathcal W$. By Property~(4), we have $H((\mathcal V\cap cv_N)\times [0,1])\subseteq\mathcal U\cap cv_N$, so $H$ homotopes $\calv\cap cv_N$ into $\calw$, staying within $\calu\cap cv_N$. As $\calw$ is contractible, this shows that the inclusion $\mathcal V\cap cv_N\hookrightarrow \mathcal U\cap cv_N$
  is nullhomotopic.

  Now suppose $U\not\in cv_N$.  By induction we may assume that the pair $(\widehat{cv_k},\widehat{cv_k}\smallsetminus cv_k)$ is LCC for all $k<N$.
  Therefore by Property~(5), there is a neighborhood $\mathcal W\subset\mathcal
  U\cap X_{S,U}$ of $U$ in $X_{S,U}$ such that $\mathcal W\cap cv_N\hookrightarrow \mathcal
  U\cap cv_N$ is nullhomotopic. By continuity of $\rho$ and $H$, there is a
  neighborhood $\mathcal V$ of $T_0$ such that $H(\mathcal V\times
  [0,1])\subseteq \mathcal U$ and $\rho(\mathcal V)\subseteq\mathcal
  W$, and by Property~(4) we have $H((\calv\cap cv_N)\times [0,1])\subseteq\calu\cap cv_N$. Therefore $H$ homotopes $\calv\cap cv_N$ into $\calw\cap cv_N$, staying within $\calu\cap cv_N$. As $\mathcal W\cap cv_N\hookrightarrow \mathcal
  U\cap cv_N$ is nullhomotopic, we deduce that the inclusion $\calv\cap cv_N\hookrightarrow\calu\cap cv_N$ is nullhomotopic.
\end{proof}

\subsection{Morphisms between trees in $\widehat{cv_N}$ and extension of the semi-flow}\label{sec-mor}

We recall from the previous section that a \emph{morphism} between two
trees $T,T'\in\widehat{cv_N}$ is a morphism between $\pi(T)$ and
$\pi(T')$ which is further assumed to be isometric and
orientation-preserving when restricted to every edge with nontrivial stabilizer. 
A morphism between two trees $T,T'\in\widehat{cv_N}$ is \emph{optimal} if the corresponding morphism from $\pi(T)$ to $\pi(T')$ is optimal. The space $\text{Mor}(\widehat{cv_N})$ of all
morphisms between trees in $\widehat{cv_N}$ is topologized by saying
that two morphisms are close whenever the corresponding morphisms
between the projections of the trees in $\overline{cv_N}$ are close,
and in addition the sources and ranges of the morphisms are close in
$\widehat{cv_N}$. As before, this space is second-countable. 
We note that the canonical map $\text{Mor}(\widehat{cv_N})\to
\text{Mor}(\overline{cv_N})$ is bounded-to-one, and it is injective
when restricted to morphisms between two fixed elements of $\widehat{cv_N}$. We denote by $\text{Opt}(\widehat{cv_N})$ the space of all optimal
morphisms between trees in $\widehat{cv_N}$. We now extend Proposition \ref{factor} to morphisms between trees in $\widehat{cv_N}$.

\begin{prop}\label{factor-morphisms}
There exist continuous maps $\widehat{H}:\text{Opt}(\widehat{cv_N})\times [0,1]\to \widehat{cv_N}$ and $\widehat\Phi,\widehat\Psi:\text{Opt}(\widehat{cv_N})\times [0,1]\to\text{Opt}(\widehat{cv_N})$ such that
\begin{itemize}
\item for all $f\in\text{Opt}(\widehat{cv_N})$ and all $t\in [0,1]$, the tree $\widehat{H}(f,t)$ is the range of the morphism $\widehat\Phi(f,t)$ and the source of the morphism $\widehat{\Psi}(f,t)$,
\item for all $f\in\text{Opt}(\widehat{cv_N})$, we have $\widehat\Phi(f,0)=id$ and $\widehat\Psi(f,0)=f$, 
\item for all $f\in\text{Opt}(\widehat{cv_N})$, we have $\widehat\Phi(f,1)=f$ and $\widehat\Psi(f,1)=id$, and
\item for all $f\in\text{Opt}(\widehat{cv_N})$ and all $t\in [0,1]$, we have $\widehat\Psi(f,t)\circ\widehat\Phi(f,t)=f$.
\end{itemize}
\end{prop} 

\begin{proof}
Let $f\in\text{Opt}(\widehat{cv_N})$ be a morphism with source $S$ and range $T$. By Proposition~\ref{factor}, the corresponding morphism from $\pi(S)$ to $\pi(T)$ factors through trees $H(f,t)$.  Notice that any morphism satisfies the hypothesis from Lemma~\ref{id-time}, i.e.\ $f$ is isometric when restricted to any arc of $\pi(T)$ with trivial stabilizer. By Lemma~\ref{id-time}, the induced morphism from $H(f,t)$ to $\pi(T)$ is isometric when restricted to any arc of $H(f,t)$ with nontrivial stabilizer. This enables us to define $\widehat{H}(f,t)$ for all $(f,t)\in\text{Opt}(\widehat{cv_N})\times [0,1]$, by pulling back the orientations on the edges of $T$ in all trees $H(f,t)$. We also get morphisms $\widehat\Phi(f,t)$ and $\widehat\Psi(f,t)$. We will check that the map $\widehat{H}$ defined in this way is continuous, from which it follows that $\widehat\Phi$ and $\widehat\Psi$ are also continuous. Since $\text{Opt}(\widehat{cv_N})$ is second-countable, it is enough to show sequential continuity.

Let $(f,t)\in\text{Opt}(\widehat{cv_N})\times [0,1]$, and let $((f_n,t_n))_{n\in\mathbb{N}}\in (\text{Opt}(\widehat{cv_N})\times [0,1])^{\mathbb{N}}$ be a sequence that converges to $(f,t)$. By Proposition \ref{factor}, the sequence $(\pi(\widehat{H}(f_n,t_n)))_{n\in\mathbb{N}}$ converges to $\pi(\widehat{H}(f,t))$. Let $\alpha\in F_N$ be an element that fixes a nondegenerate arc $e=[u,v]$ in $\widehat{H}(f,t)$. We denote by $l(e)$ the length of $e$. Then $\alpha$ also fixes a nondegenerate arc in the range $T$ of $f$. Let $C:=BBT(f)$ be the bounded backtracking constant of $f$ (whose definition is recalled in the paragraph preceding Lemma~\ref{BBT}). Recall from Lemma~\ref{BBT} that $BBT(\Psi(f,t))\le BBT(f)$ for all $t\in [0,1]$. 

We observe that there exists a legal segment $I$ of length larger than $2BBT(f)+l(e)$, centered at the midpoint of $e$. Indeed, every segment contained in one of the subtrees with dense orbits of the Levitt decomposition is legal (see e.g.\ \cite[Lemma~2.2]{BGH}), and since $\Psi(f,t)$ is optimal, every legal segment ending at a vertex of $\pi(\widehat{H}(f,t))$ can be enlarged to a longer legal segment. We can then find an element $\beta\in F_N$ which is hyperbolic in $\widehat{H}(f,t)$ and whose axis contains $I$ (this is obtained by multiplying two hyperbolic elements whose axes are disjoint, and such that the bridge between their axes contains $I$). This is illustrated in Figure~\ref{fig-bbt}. 

\begin{figure}
\begin{center}
\def\JPicScale{.8}
\input{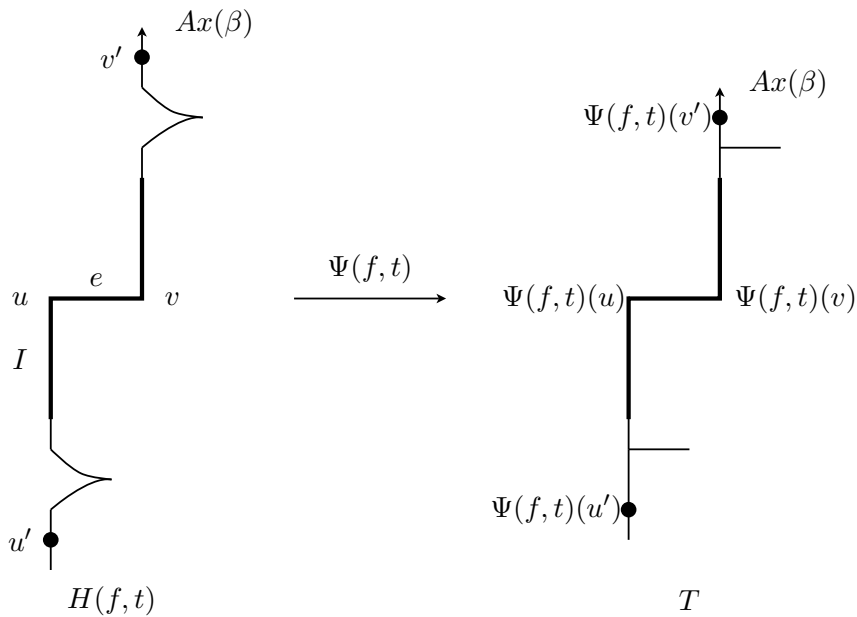}
\caption{The situation in the proof of Proposition~\ref{factor-morphisms}.}
\label{fig-bbt}
\end{center}
\end{figure}

 We claim that $\beta$ is hyperbolic in $T$, that its axis crosses the image of $e$, and that the relative orientations of $(\alpha,\beta)$ are the same in $\widehat{H}(f,t)$ and in $T$. To prove this, we first observe that there is no point $u'$ on the axis of $\beta$ in $H(f,t)$ which has the same image as $u$ in $T$: otherwise, our choice of $I$ would imply that the image of $[u,u']$ in $T$ contains a point at distance larger than $BBT(f)$ from $\Psi(f,t)(u)=\Psi(f,t)(u')$, contradicting the definition of the BBT. Similarly, there is no point $v'$ on the axis of $\beta$ in $H(f,t)$ which has the same image as $v$ in $T$. In particular, if $u'$ and $v'$ are two points lying on the axis of $\beta$ in $H(f,t)$, sufficiently far from $u$ and $v$, and such that $u',u,v,v'$ are aligned in this order in $H(f,t)$, then their images in $T$ are also aligned in this order. Using this observation, we can thus find three points of the form $u,\beta u$ and $\beta^2u$ lying on the axis of $\beta$ in $H(f,t)$, whose images in $T$ are still aligned in this order. This implies that $\beta$ is hyperbolic in $T$, and that the axis of $\beta$ in $T$ crosses the image of $e$, and the relative orientations of $(\alpha,\beta)$ in $\widehat{H}(f,t)$ and in $T$ are the same, as claimed. 

For all $n\in\mathbb{N}$, let $T_n$ be the range of the morphism $f_n$. Since $(T_n)_{n\in\mathbb{N}}$ converges to $T$ in $\widehat{cv_N}$, the relative orientations of $(\alpha,\beta)$ in $T_n$ and in $T$ eventually agree. Since the characteristic sets of $\alpha$ and $\beta$ have nondegenerate overlap in all trees $\widehat{H}(f,t')$ with $t'\ge t$, a compactness argument shows that for $n\in\mathbb{N}$ large enough, the characteristic sets of $\alpha$ and $\beta$ have nondegenerate overlap in all trees $\widehat{H}(f_n,t')$ with $t'\ge t_n$. Therefore, the relative orientation of $(\alpha,\beta)$ cannot change along the path from $\widehat{H}(f_n,t_n)$ to $T_n$, so it is the same in $\widehat{H}(f_n,t_n)$ and in $T_n$. This implies that the relative orientations of $(\alpha,\beta)$ eventually agree in $\widehat{H}(f_n,t_n)$ and in $\widehat{H}(f,t)$. Lemma \ref{proving-cv} then shows that $\widehat{H}(f_n,t_n)$ converges to $\widehat{H}(f,t)$. 
\end{proof}

\subsection{Approximations by trees that split over free splittings}\label{sec-approx-split}

The following lemma is a version of Lemma \ref{fold} for $\widehat{cv_N}$, which easily follows from the version in $\overline{cv_N}$.

\begin{lemma}\label{fold-2}
Let $T\in\widehat{cv_N}$, and let $\calu$ be an open neighborhood of $T$ in $\widehat{cv_N}$. Then there exists an open neighborhood $\calw\subseteq\calu$ of $T$ in $\widehat{cv_N}$ such that if $U\in\calw$ is a tree that admits a morphism $f$ onto $T$, and if $U'\in\widehat{cv_N}$ is a tree such that $f$ factors through morphisms from $U$ to $U'$ and from $U'$ to $T$, then $U'\in\calu$.
\qed
\end{lemma}

A \emph{one-edge free splitting} of $F_N$ is the Bass--Serre tree of a graph of groups decomposition of $F_N$, either as a free product $F_N=A\ast B$, or as an HNN extension $F_N=A\ast$. If $T\in\widehat{cv_N}$ is a tree that splits as a graph of actions over a splitting $S$ of $F_N$, we say that an attaching point of a vertex action $T_v$ is \emph{admissible} if either it belongs to the $G_v$-minimal subtree of $T_v$, or else it is an endpoint of an arc with nontrivial stabilizer contained in $G_v$. The following proposition extends the analogous result for $\overline{cv_N}$ (see \cite[Theorems~3.6 and~3.11]{Hor14-1}).

\begin{prop}\label{approx}
Let $T\in\widehat{cv_N}$, and let $\calu$ be an open neighborhood of
$T$. Then there exists a tree $U\in\calu$ that splits as a graph of actions over a one-edge free splitting $S$ of $F_N$ with admissible attaching points, 
coming with an
optimal morphism $f:U\to T$ such that  
\begin{itemize}
\item every arc with nontrivial stabilizer in $T$ is the $f$-image of an arc with nontrivial stabilizer in $U$, 
\item the $f$-preimage of any arc with nontrivial stabilizer in $T$ consists of a single arc with nontrivial stabilizer in $U$, 
\item $f$ is an isometry when restricted to the minimal subtree of any vertex action.
\end{itemize}
\end{prop}

\begin{proof}
The proof of Proposition
\ref{approx} builds on classical approximation techniques of very
small $F_N$-trees, we present a sketch of the argument. Let $X$ be the
underlying graph of groups of the Levitt decomposition of $T$. We can
assume that all edge stabilizers of $X$ are nontrivial, otherwise the
result is obvious by choosing $U=T$ and $f$ to be the identity. We may
do the same if $T$ is simplicial (with all edge groups nontrivial) as
such trees are graphs of actions with skeleton a 1-edge free
splitting (see \cite{She,Swa} or \cite[Lemma~4.1]{BF94}).
We now assume $T$ is not simplicial.

If some vertex action $T_v$ of the Levitt decomposition is nongeometric (i.e.\ not dual to a measured
foliation on a $2$-complex \cite{LP97}), then one can approximate $T_v$
by a geometric tree $U'_v\in\calu$ that contains a simplicial edge $e$
with trivial stabilizer and admits an optimal
morphism $f_v:U'_v\to T_v$, keeping point stabilizers the same; these induce an optimal morphism $f':U'\to T$. Let $U'$ be the tree defined as a graph of actions over $X$, where $T_v$ is replaced by $U'_v$ (the attaching points are prescribed by the stabilizers because all edge stabilizers in $X$ are nontrivial). The edge $e$ is dual to a one-edge free splitting $S$
of $F_N$, and $U$ splits as a graph of actions over $S$. Up to slightly folding $e$ (and slightly increasing its
length), we can assume that attaching points are admissible.
By folding within every vertex action $A\actson U'_A$ of $U'$, the morphism $f'$ factors through a tree $U$ in such a way that the induced morphism from $U$ to $T$ is an isometry when restricted to the $A$-minimal subtree of $U$. In view of Lemma~\ref{fold-2}, we can also ensure that $U\in\calu$.

We now assume that $T\in\widehat{cv_N}$ is geometric. Every geometric
$F_N$-tree splits as a graph of actions with indecomposable vertex
actions, which are either dual to arational measured foliations on
surfaces, or of Levitt type.

If $T$ contains a Levitt component, we approximate it by free and
simplicial actions by first running the Rips machine, and then cutting
along a little arc transverse to the foliation in a naked band, see
\cite{BF94,Gui98}: this remains true in
$\widehat{cv_N}$ because each of these indecomposable trees has
trivial arc stabilizers. This gives an approximaton of $T$ by a tree $U$ that splits as a graph of actions over a one-edge free splitting of $F_N$, and comes equipped with an optimal morphism onto $T$. We then get the required conditions on $f$ as in the nongeometric case.
 
Finally, one is left with the case where $T$ is geometric, and all its
indecomposable subtrees are dual to measured foliations on surfaces. In this
situation, it follows from \cite[Proposition~5.10]{Hor14-2} that either $T$ splits as a graph of actions over
a one-edge free splitting of $F_N$, or else some indecomposable subtree from the geometric decomposition has an \emph{unused}
boundary component (this means that in the skeleton of the decomposition of $T$ into its indecomposable geometric components, the boundary curve is not attached to any incident edge). In the latter situation, we can again cut along a
little arc with extremity on the unused boundary component, and
transverse to the foliation, to approximate the indecomposable
component. The required condition on $f$ is obtained as above.
\end{proof}

\subsection{Definition of $X_{S,U}$, construction of the deformation, and proof of Theorem \ref{strategy}}\label{sec-construction}

In this section, given $T_0\in\widehat{cv_N}$ and a neighborhood
$\mathcal U$ we will define a deformation $H$ of $\widehat{cv_N}$ onto
a particularly nice set of trees (the set $X_{S,U}$ in Theorem~\ref{strategy}). Here, $S$ will be a 1-edge
free splitting of $F_N$ and all trees in $X_{S,U}$ will split as graphs of
actions over the splitting $S$. The tree $U$, which comes from Proposition~\ref{approx}, will be near $T_0$ and will be the image of $T_0$ under
the deformation.  
The construction of $X_{S,U}$ and the deformation (without Property~(5) from Theorem~\ref{strategy})
work equally well in $\overline{cv_N}$. 

From now on, we fix a tree $T_0\in\widehat{cv_N}$, and an open
neighborhood $\calu$ of $T_0$ in $\widehat{cv_N}$. Let
$\calw\subseteq\calu$ be a smaller neighborhood of $T_0$ provided by
Lemma~\ref{fold-2}, i.e.\ such that whenever a morphism from $U\in\calw$ to $T_0$ factors through a tree $U'$, then $U'\in\calu$. We choose a one-edge free splitting $S$ of $F_N$,
and a tree $U\in\calw$ that splits as a graph of actions over $S$ with
admissible attaching points, provided by Proposition \ref{approx}. In
particular, there is an optimal morphism $f:U\to T_0$ such that for every vertex group $A$ of $S$ with vertex action $U_A$, the restriction of $f$ to the $A$-minimal subtree $U_A^{\min}\subseteq U_A$ is an isometry. 

The description of the
deformation depends on whether the splitting $S$ is $F_N=A\ast B$ or
$F_N=A\ast_1$. We will concentrate on the case of a free product, and
only briefly explain how to adapt to the case of an HNN extension, as
the details differ only in notation. 

\paragraph*{Case of a free product.}

There will be several cases in the construction, depending whether the attaching point
$u_A\in U$ of the vertex action $A\actson U_A$ belongs to $U_A^{\min}$
or not (and similarly for $u_B$). The three possible cases are illustrated in Figure~\ref{fig-cases}.

\begin{figure}
\centering
\input{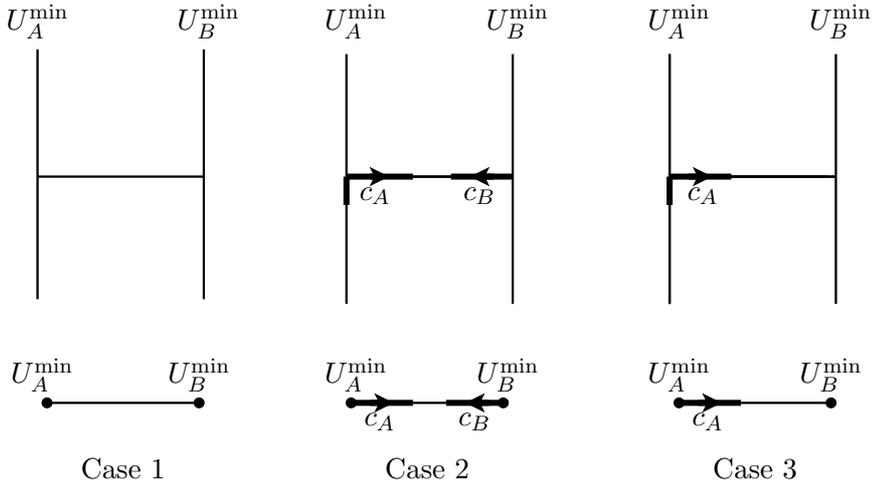}
\caption{The tree $U$ in all three cases when the splitting $S$ is a free product.}
\label{fig-cases}
\end{figure}

~\\

\noindent \textbf{Case 1:} \textit{The free splitting $S$ is the
  Bass--Serre tree of a decomposition of $F_N$ as a free product
  $F_N=A\ast B$, and both $u_A\in U_A^{\min}$ and $u_B\in U_B^{\min}$.}
~\\
\\
\textbf{Case 1.1:} \textit{Both $A$ and $B$ have rank at least $2$.}
\\
\\
Corollary~\ref{basepoint} yields a continuous map
$b_A:\widehat{cv(A)}\to\text{Map}(A,\widehat{cv(A)})$, such that
$b_A(U_A^{\min})(\ast)=u_A$.
We will abuse notation and write
$b_A(T_A)$ for $b_A(T_A)(\ast)$.
Similarly, we have a map $b_B$.
We make the following definition.

\begin{de}\label{def-X}
Under the assumptions of Case 1.1, the set $X_{S,U}\subseteq\widehat{cv_N}$ is the set of all trees $T$ that split as a graph of actions over $S$, with the condition that both vertex actions $T_A$ and $T_B$ are minimal, and the attaching point on the $A$-side (resp.\ on the $B$-side) is $b_A(T_A)$ (resp.\ $b_B(T_B)$).
\end{de}

Notice that we allow the possibility that $T_A$ or $T_B$ (or both) is the trivial tree.

\begin{prop}\label{def-retr}
Let $T_0\in\widehat{cv_N}$, let $\calu$ be an open neighborhood of
$T_0$ in $\widehat{cv_N}$, and let $S,U$ be chosen as above, and
assume Case 1.1. 
\\ Then there exist a continuous map $\rho:\widehat{cv_N}\to X_{S,U}$ with $\rho(T_0)=U$, and a homotopy $H:\widehat{cv_N}\times
[0,1]\to\widehat{cv_N}$ from the identity to $\rho$, such that 
$H(\{T_0\}\times [0,1])\subseteq\mathcal U$ and $H(cv_N\times [0,1])\subseteq
cv_N$.
\end{prop}

We will start by defining $\rho(T)$ for all $T\in\widehat{cv_N}$. Let $T_A$ be the $A$-minimal subtree of $T$, and let $T_B$ be the $B$-minimal subtree of $T$. The tree $\rho(T)$ is defined as the tree that splits as a graph of actions over $S$ with vertex actions $T_A$ and $T_B$, and attaching points $b_A(T_A)$ and $b_B(T_B)$, where the simplicial edge from the splitting is assigned length $d_T(b_A(T_A),b_B(T_B))$. 

Let now $f_T:\rho(T)\to T$ be the morphism which is the identity when restricted to $T_A$ or $T_B$, and sends the simplicial edge linearly onto its image
in $T$. Observe that by construction
we have $\rho(T_0)=U$ and $f_{T_0}=f$.

\begin{lemma}\label{f-cont}
Under the assumptions of Case 1.1, the morphism $f_T:\rho(T)\to T$
depends continuously on the tree $T\in\widehat{cv_N}$.
\end{lemma}

\begin{proof}
We need to establish that the tree $\rho(T)$ depends continuously on
$T$. Once this is established, continuity of $T\mapsto f_T$ follows
from the construction. Continuity of $\pi(\rho(T))\in\overline{cv_N}$
is obvious from the continuity of the choice of basepoints.

Now suppose that $T_n\to T$ and notice that all nontrivial arc
stabilizers in $\rho(T)$ are conjugate into either $A$ or $B$. By
Lemma \ref{proving-cv}, we only need to associate to any element
$\alpha\in F_N$ fixing a nondegenerate arc in $\rho(T)$, an element
$\beta\in F_N$ that is hyperbolic in $\rho(T)$ and whose axis has
nondegenerate intersection with $\text{Char}_T(\alpha)$, such that the
relative orientations of $(\alpha,\beta)$ eventually agree in
$\rho(T_n)$ and in $\rho(T)$. The element $\alpha$ fixes a
nondegenerate arc in either the $A$- or the $B$-minimal subtree of
$T$. As the minimal $A$-invariant subtree of $T_n$ converges (as an $A$-tree in $\widehat{cv(A)}$) to the minimal $A$-invariant subtree of $T$, we are done by choosing $\beta$ to be in $A$ or $B$.
\end{proof}

\begin{proof}[Proof of Proposition \ref{def-retr}]  
For all $T\in\widehat{cv_N}$ and all $t\in [0,1]$, we let
$H(T,t):=\widehat{H}(f_T,1-t)$, with the notation from Proposition
\ref{factor-morphisms}. Since $\rho(T_0)=U$ belongs to $\calw$,
Lemma~\ref{fold-2} implies that $\widehat{H}(f_{T_0},t)\in\calu$ for
all $t\in [0,1]$. In addition, we have $\widehat{H}(f_T,0)=\rho(T)$ and
$\widehat{H}(f_T,1)=T$ for all $T\in\widehat{cv_N}$, and
$\widehat{H}(f_T,t)\in cv_N$ for all $(T,t)\in cv_N\times
[0,1]$ in view of Remark~\ref{rk}. 
\end{proof}

\begin{proof}[Proof of Theorem \ref{strategy} in Case 1.1]
Properties~(1) to~(4) have been established in Proposition~\ref{def-retr}, so there remains to establish Property~(5). Assume that $(\widehat{cv_k},\widehat{cv_k}\smallsetminus cv_k)$ is LCC  for all $k<N$. We will prove that
  $(X_{S,U},X_{S,U}\smallsetminus cv_N)$ is LCC at every point of
  $X_{S,U}\smallsetminus cv_N$.

  A tree $T$ in $X_{S,U}$ is completely
  determined by $T_A,T_B$ and the length of the arc connecting the
  basepoints $b_A(T_A)$ and $b_A(T_B)$. This gives a canonical homeomorphism
  $$X_{S,U}\cong \widehat{cv(A)}\times \widehat{cv(B)}\times
  [0,\infty)$$
  Let $(T_A,T_B,d)$ be a point in the product space and $\mathcal U$ a
  given neighborhood of it. After shrinking $\mathcal U$ we may assume
  it has the form $$\mathcal U=\mathcal U_A\times\mathcal U_B\times
  J$$
  where the three factors are open in their respective spaces and $J$
  is an interval. 
  
 We claim that there exists a  neighborhood $\calv_A\subseteq\calu_A$ of $T_A$ in $\widehat{cv(A)}$ such that the inclusion $\mathcal{V}_A\cap cv(A)\hookrightarrow\calu_A\cap cv(A)$ is nullhomotopic. Indeed, if $T_A\in cv(A)$, this follows from local contractibility of $cv(A)$. If $T_A\notin cv(A)$, this follows from our induction hypothesis which ensures that $(\widehat{cv(A)},\widehat{cv(A)}\smallsetminus cv(A))$ is LCC. Similarly, there exists a neighborhood $\calv_B\subseteq\calu_B$ such that the inclusion $\calv_B\cap cv(B)\hookrightarrow\calu_B\cap cv(B)$ is nullhomotopic. 
 The neighborhood $$\mathcal V= \mathcal V_A\times\mathcal V_B\times
  J$$
  of $T$ is then such 
  that the inclusion $\mathcal V\cap cv_N\hookrightarrow \mathcal U\cap cv_N$ is
  nullhomotopic (notice indeed that $\calv\cap cv_N=(\calv_A\cap cv(A))\times (\calv_B\cap cv(B))\times J$).
  \end{proof}
~\\
\noindent \textbf{Case 1.2:} \textit{We have $\mathrm{rk}(A)\ge 2$ and $\mathrm{rk}(B)=1$.}
\\
\\
In this case, we have a map $b_A$ as in Case~1.1.

\begin{de}\label{def-X-12}
Under the assumptions of Case 1.2, the set $X_{S,U}\subseteq\widehat{cv_N}$ is the set of all trees $T$ that split as a graph of actions over $S$, with the condition that both vertex actions $T_A$ and $T_B$ are minimal, and the attaching point on the $A$-side is $b_A(T_A)$.
\end{de}

Notice that there is no condition on the attaching point on the $B$-side. Indeed, if $T_B$ is reduced to a point, then there is a unique way of attaching. Otherwise it is a line, and the tree $T$ does not depend on a choice of attaching point in $T_B$. 

Using Corollary~\ref{basepoint}, we also get a continuous map
$b_B:\widehat{cv_N}\to\text{Map}(F_N,\widehat{cv_N})$ such that for all
$T\in\widehat{cv_N}$, the range of $b_B(T)$ is $T$, and  
$b_B(T_0)(\ast)=u_B$, and $b_B(T)(\ast)$ is contained in the characteristic set of $B$.

Given $T\in\widehat{cv_N}$, we define $\rho(T)$ as follows. Let $T_A$ be the $A$-minimal subtree of $T$, and let $T_B$ be either the trivial $B$-action on a point if $B$ is elliptic in $T$, or else the $B$-minimal subtree of $T$. The tree $\rho(T)$ is defined as the tree that splits as a graph of actions over $S$ with vertex actions $T_A$ and $T_B$, with attaching point $b_A(T_A)$ on the $A$-side, where the simplicial edge from the splitting is assigned length $d_T(b_A(T_A),b_B(T))$. 

Let now $f_T:\rho(T)\to T$ be the morphism which is the identity when restricted to $T_A$ or $T_B$, and sends the simplicial edge linearly onto the segment $[b_A(T_A),b_B(T)]$. Observe that by construction
we have $\rho(T_0)=U$ and $f_{T_0}=f$.

Then $\rho(T)$ and $f_T$ depend continuously on $T$. The rest of the proof is identical to Case~1.1.
~\\
\\
\noindent\textbf{Case 1.3:} \textit{We have $N=2$, and the free splitting $S$ is the Bass--Serre tree
of a decomposition of $F_N$ as a free product $F_N=A\ast B$ with $\mathrm{rk}(A)=1$ and $\mathrm{rk}(B)=1$.}
\\
\\
The argument is then exactly the same as in Case 1.2 (where both sides are treated as the $B$-side in the argument).

~\\
\\
\noindent \textbf{Case 2:} \textit{The free splitting $S$ is the Bass--Serre tree
of a decomposition of $F_N$ as a free product $F_N=A\ast B$ and
$u_A\not\in U_A^{\min}$, $u_B\not\in U_B^{\min}$.}

~\\

According to Proposition \ref{approx}, $u_A$ is an endpoint of an arc
$I_A\subset U$ with nontrivial stabilizer $\langle c_A\rangle$ with
$c_A\in A$. We choose $c_A$ so that the orientation of the
characteristic set of $c_A$ is pointing away from $U_A^{\min}$. Note that $I_A\cap U_A^{\min}$ is a (perhaps degenerate)
subarc of $I_A$ containing the other endpoint of $I_A$. We define
$c_B$ similarly. Notice that in this case the minimality of $U$ implies that both $A$ and $B$ have rank at least $2$.

\begin{de}
Under the assumptions of Case~2, the set $X_{S,U}\subseteq \widehat{cv_N}$ is the set of all trees 
$T$ that split as graphs of actions over the splitting $S$ and with
attaching points in the vertex trees $T_A,T_B$ belonging to the
characteristic sets of $c_A$ and $c_B$ respectively. 
\end{de}

Next, for $T\in\widehat{cv_N}$ define the basepoint $b_A'(T)\in Char_T(c_A)$ as
the projection of $Char_T(c_B)$ (if the two are disjoint) or the
midpoint of the overlap with $Char_T(c_B)$ (if they are not
disjoint). It follows from Corollary~\ref{basepoint} that this basepoint varies continuously with $T$. Similarly define $b_B'(T)\in
Char_T(c_B)$. Let $T_A$ be the smallest $A$-invariant subtree of $T$ that contains both
$T_A^{\min}$ and $b_A'(T)$, and similarly define $T_B$.

We now define $\rho:\widehat{cv_N} \to X_{S,U}$. 
Given $T\in\widehat{cv_N}$, we let $\rho(T)$ be the tree that splits as a graph of actions over the free splitting
$S$, with vertex actions $T_A$ and $T_B$, with attaching points
$b_A'(T),b_B'(T)$, and with arc length $d_T(b_A'(T),b_B'(T))$. We also
have the obvious morphism $f_T:\rho(T)\to T$. Note that $\rho(T_0)=U$.

\begin{lemma}\label{f-cont-2}
Under the assumptions of Case 2, the morphism $f_T:\rho(T)\to T$
depends continuously on $T$.
\end{lemma}

\begin{figure}
\centering
\input{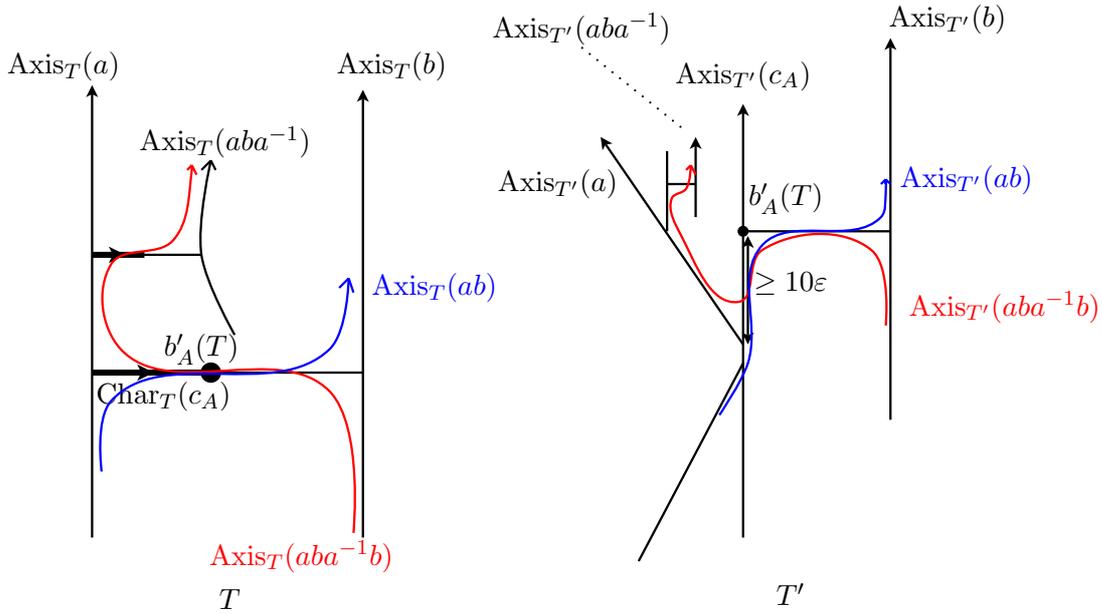}
\caption{The tree $T$ and a tree $T'$ in the neighborhood of $T$ in the last case of the proof of Lemma~\ref{f-cont-2}. In the picture, we have represented the situation where $c_A$ becomes hyperbolic in $T'$.}
\label{fig-f-cont}
\end{figure}

\begin{proof}
Again we need to establish that the tree $\rho(T)$ depends continuously on
$T$. 
Continuity of $\pi(\rho(T))\in\overline{cv_N}$
follows from the construction and Guirardel's Reduction Lemma from
\cite[Section 4]{Gui98}.

Notice that all nontrivial arc stabilizers in $\rho(T)$ are conjugate into either $A$ or $B$. By Lemma~\ref{proving-cv}, in order to complete the proof of the continuity of $\rho$, we only need to associate to any element $\alpha\in F_N$ fixing a nondegenerate arc in $\rho(T)$, an element $\beta\in F_N$ that is hyperbolic in $\rho(T)$ and whose axis has nondegenerate intersection with $\text{Char}_T(\alpha)$, such that the relative orientations of $(\alpha,\beta)$ eventually agree in $\rho(T_n)$ and in $\rho(T)$. This is clear if $\alpha$ fixes a nondegenerate arc in either the $A$- or the $B$-minimal subtree of $T$ (see the proof of Lemma~\ref{f-cont}). 

If these minimal subtrees have empty intersection in $\rho(T)$, and $\alpha$ fixes a nondegenerate arc -- say $\alpha=c_A$ -- on the segment that joins them in $T$, then one can choose for $\beta$ an element which is a product $ab$, where $a\in A$ and $b\in B$ are hyperbolic in $T$. Indeed (see Figure~\ref{fig-f-cont}), for every $\epsilon>0$ small enough, we can find a neighborhood of $T$ such that for all trees $T'$ in this neighborhood, the intersection of the characteristic sets of $c_A$ and $c_B$ has size at most $\epsilon$, and the distance between $b'_A(T)$ and the projection of the axis of $b$ to $\Char_{T'}(c_A)$ is at most $\epsilon$, while the distance between $\Char_{T'}(a)$ and $\Char_{T'}(b)$ is at least $10\epsilon$. In addition, the characteristic sets of $c_A$ and of $baba^{-1}$ have nondegenerate intersection whose orientations disagree in a neighborhood of $T$. This implies that for all trees $T'$ in a neighborhood of $T$, the point $b'_A(T)$ is at positive distance from $\Char_{T'}(a)\cap\Char_{T'}(c_A)$ (or from the bridge between these two characteristic sets in case their intersection is empty). It then follows that the characteristic sets of $c_A$ and $ab$ have nondegenerate intersection with agreeing orientations in a neighborhood of $T$.
\end{proof}

\begin{proof}[Proof of Theorem \ref{strategy} in Case 2]
Using Lemma~\ref{f-cont-2}, Properties~(1) to~(4) are proved in the same way as in Case~1 (Proposition~\ref{def-retr} stays valid in Case~2 with the same proof). It remains to prove Property~(5). First, using the first part of Corollary~\ref{basepoint}, we can choose continuous basepoints
  $b_A:\widehat{cv(A)}\to Map(A,\widehat{cv(A)})$ and
  $b_B:\widehat{cv(B)}\to Map(B,\widehat{cv(B)})$ that belong to
  characteristic sets of $c_A$ and $c_B$ respectively (Corollary~\ref{basepoint} was stated for $\overline{cv_N}$ but
  orientations of intervals just carry along). 
 Then construct a map
  $$\Phi:X_{S,U}\to \widehat{cv(A)}\times \widehat{cv(B)}\times
  (-\infty,\infty)\times (-\infty,\infty)\times [0,\infty)$$ by
    sending $T\in X_{S,U}$ to $(T_A^{\min},T_B^{\min},l_A,l_B,d)$
    where    
    $l_A$ is the signed distance, measured along the oriented characteristic set of $c_A$, from $b_A(T_A)$ (viewed as a point in $T$) to $b_A'(T)$, and
    similarly for $l_B$, and $d$ is the length of the attached
    arc. This map is continuous because the various basepoints vary continuously with $T$ and $d=d_T(b'_A(T),b'_B(T))$. 
    
    We claim that the map $\Phi$ is a homeomorphism. Once this claim is established, the argument that $(X_{S,U},X_{S,U}\smallsetminus cv_N)$ is LCC is similar to the proof of Case~1. To prove the claim, we
    construct the inverse $\Psi$. Given
    $(T_A,T_B,l_A,l_B,d)$, where $T_A$ (resp.\ $T_B$) is a minimal $A$-tree (resp.\ $B$-tree), construct $b_A'(T)$ by taking
    the point in $Char_{T_A}(c_A)$ (which is oriented) at the correct
    signed distance of $b_A(T_A)$. More precisely, if $c_A$ is hyperbolic, then the point $b_A'(T)$ is
    uniquely defined. If $c_A$ is elliptic, it may happen that $|l_A|$ is larger
    than the distance between $b_A(T_A)$ and the suitable endpoint of
    $Fix_{T_A}(c_A)$. In that case attach an arc to this endpoint and
    add it to the fixed set of $c_A$ (and equivariantly attach an
    orbit of arcs). Its orientation is determined by the sign of
    $l_A$. This yields a tree $T_A'$. Similarly construct $T_B'$ (in the case where $c_A$ is hyperbolic, we just let $T_A'=T_A$, and likewise on the $B$-side). Then
    $\Psi(T_A,T_B,l_A,l_B,d):=T$ is the graph of actions
    obtained by gluing an arc of length $d$ between $b_A'(T_A)$ and
    $b_B'(T_B)$. By construction $\Psi$ is an inverse of $\Phi$. 
    
    We claim that the map $(T_A,l_A)\mapsto (T_A',b'_A(T_A))$ is continuous; continuity of $\Psi$ will then follow from Guirardel's Reduction Lemma \cite[Section~4]{Gui98}. To prove the claim, it is enough to observe that if $(T_A^n,l_A^n)$ converges to $(T_A,l_A)$, then for all $g\in F_N$, the distance $d_{(T'_A)^{n}}(b'_A(T_A^n),gb'_A(T_A^n))$ converges to $d_{T'_A}(b'_A(T_A),gb'_A(T_A))$. This is proved by using the analogous convergence for the basepoints $b_A$ instead of $b'_A$, and the fact that the basepoints $b_A$ and $b'_A$ always lie at the same distance $l_A$ on the characteristic set of $c_A$.
    \end{proof}

\begin{rk}
Notice that the above proof of Property~(5) does not work in $\overline{cv_N}$ where we do not have a well-defined notion of signed distance along the characteristic set of an element of $F_N$. 
\end{rk}

~\\

\noindent \textbf{Case 3:} \textit{The free splitting $S$ is the Bass--Serre tree
of a decomposition of $F_N$ as a free product $F_N=A\ast B$ and
$u_A\in U_A^{\min}$, $u_B\not\in U_B^{\min}$.}

~\\

Fix a nontrivial element $g_A\in A$. In this hybrid case we define basepoints $b_A(T)$ as in Case 1, using
Lemma~\ref{basepoint}, and
$b'_B(T)$ as in Case 2, projecting the characteristic set of $g_A$ to $Char_T(c_B)$ (or
taking the midpoint of the overlap). 
The argument in this case follows Cases 1 and 2 and is
left to the reader. 
  
\paragraph*{Case of an HNN extension.}

We now briefly explain how to adapt the above construction in the case of
an HNN extension. As the argument is the same as in the free product case (details differ only in notation), we leave the
proof of Proposition \ref{def-retr} to the reader in this case.

\begin{figure}
\begin{center}
\def\JPicScale{.7}
\input{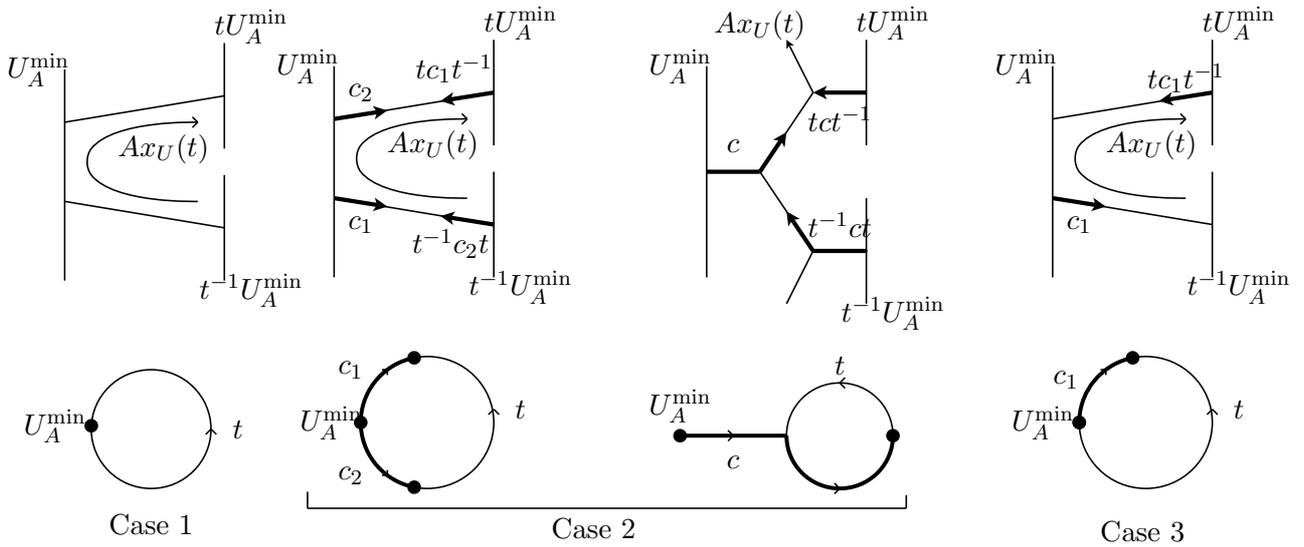}
\caption{The tree $U$ in Case 2 of the proof of Proposition \ref{def-retr}, and the quotient graph of groups.}
\label{fig-splitting}
\end{center}
\end{figure}

We first assume that $N\ge 3$. The tree $U$ has one of the shapes represented in Figure~\ref{fig-splitting}, where $t$ denotes a stable letter for the splitting; this leads to three cases as in the case of a free product. There are two attaching points in $U_A$, which we denote by $u_1$ and $u_2$. 

The first case is the case where both $u_1$ and $u_2$ belong to $U_A^{\min}$. As $N\ge 3$, the rank of $A$ is at least $2$. By Corollary~\ref{basepoint}, there exist continuous maps $b_1:\widehat{cv(A)}\to\text{Map}(A,\widehat{cv(A)})$ and $b_2:\widehat{cv(A)}\to\text{Map}(A,\widehat{cv(A)})$ such that
$b_1(U_A^{\min})(\ast)=u_1$ and $b_2(U_A^{\min})(\ast)=u_2$. We then define $X_{S,U}\subseteq\widehat{cv_N}$ to be the set of all trees $T$ that split as a graph of actions over $S$, with the condition that
both vertex actions $T_A$ and $T_B$ are minimal, 
and the attaching
points are $b_1(T_A)$ and $b_2(T_B)$ respectively. We then argue as in Case~1 of the free product case.

The second case is when neither $u_1$ nor $u_2$ belongs to $U_A^{\min}$. Then $u_1$ and $u_2$ belong to arcs stabilized by elements $c_1,c_2\in F_N$ (in the case where $c_1=c_2$, the points $u_1$ and $u_2$ may be the two extremities of the segment fixed by $c_1=c_2$, but it might also happen that one of them is not an extremity, see Case~2 in Figure~\ref{fig-splitting}). We then define a basepoint $b'_1(T)$ to be the projection of $\text{Char}_{T}(t^{-1}c_2t)$ to $\text{Char}_{T}(c_1)$ if these sets have nondegenerate intersection, or else the midpoint of their intersection. Similarly, we let $b'_2(T)$ be the projection of $\text{Char}_{T}(tc_1t^{-1})$ to $\text{Char}_{T}(c_2)$ if these sets have nondegenerate intersection, or else as the midpoint of their intersection. We then define $X_{S,U}\subseteq \widehat{cv_N}$ to be the set of all trees 
$T$ that split as graphs of actions over the splitting $S$ and with
attaching points belonging to the characteristic sets of $c_1$ and $c_2$ respectively. We argue as in Case~2 of the free product case.

Finally, in the hybrid case where exactly one of the points $u_1,u_2$ belongs to $U_A^{\min}$, we define one basepoint as in Case~1, and one basepoint as in Case~2.

We now assume that $N=2$, so $A=\langle a\rangle$ is isomorphic to $\mathbb{Z}$. In this case, we define $X_{S,U}$ as the set of all trees $T$ that split as a graph of actions over $S$, where the vertex action is either a tree in $\widehat{cv_1}$ or an oriented arc stabilized by $a$. There is a homeomorphism between $X_{S,U}$ and $\widehat{cv_1}\times\mathbb{R}\times\mathbb{R}_+$, a point in $X_{S,U}$ being completely determined by the data of a tree in $\widehat{cv_1}$, together with the signed distance between the two attaching points (when the tree is $\widehat{cv_1}$ is trivial, this is the signed length of the arc fixed by $a$), and the length of the arc coming from $S$.

\appendix
\section{A \v{C}ech homology lemma}\label{sec-Cech}

We regard the $n$-sphere $S^n$ as the boundary of the unit ball
$B^{n+1}$ in $\R^{n+1}$ and $S^{n-1}\subset S^n$ as the equator
$S^n\cap (\R^n\times\{0\})$. The northern (resp.\ southern) hemisphere
$B^n_+$ (resp.\ $B^n_-$) is the set of points in $S^n$ whose last
coordinate is nonnegative (resp.\ nonpositive).  We will say that a map
$h:S^n\to\R$ is {\it standard} if $h^{-1}(0)=S^{n-1}$,
$h^{-1}([0,\infty))=B^n_+$ and $h^{-1}((-\infty,0])=B^n_-$. The reader
is refered to \cite[Chapters IX,X]{eilenberg-steenrod} for basic facts
about \v{C}ech homology.

\begin{lemma}\label{Cech}
Let $n\ge 1$, and let $\tilde h:B^{n+1}\to\R$ be a continuous map whose restriction to $S^n=\partial B^{n+1}$ is standard. Then the inclusion
$$S^{n-1}\hookrightarrow Y:=\tilde h^{-1}(0)$$ is trivial in \v{C}ech
homology $\check H_{n-1}$ with ${\mathbb Z}/2$-coefficients (one has
to consider reduced homology if $n=1$).
\end{lemma}

\begin{proof}
The proof is illustrated in Figure \ref{fig-coho}. 
Represent $Y$ as a nested intersection $\cap_i K_i$ of compact
neighborhoods. 
By the continuity
of \v{C}ech homology \cite[Theorem X.3.1]{eilenberg-steenrod}, we have
$$\check H_{n-1}(Y)=\underset{\leftarrow}{\lim}~ H_{n-1}(K_i),$$
so it suffices to show that $S^{n-1}\hookrightarrow K_i$ is trivial in
(singular) $H_{n-1}$ for every $i$. 

Fix one such $K_i$ and choose $\epsilon>0$ so that
 $\tilde h^{-1}(-2\epsilon,2\epsilon)\subset K_i$.  Let $g:B^{n+1}\to\R$ be a
smooth map such that $|\tilde{h}(x)-g(x)|<\epsilon/2$ for every $x$ and so
that $g$ and $g_0=g|S^n$ have $\epsilon$ as a regular value. Thus
$V=g_0^{-1}(\epsilon)\subset B_+^n$ (in blue on the figure) is an $(n-1)$-manifold and it
bounds the $n$-manifold $g^{-1}(\epsilon)\subset
\tilde h^{-1}(\epsilon/2,3\epsilon/2)\subset K_i$ (in red on the figure). On the other hand, $V$ is
disjoint from $S^{n-1}$ and $V\sqcup S^{n-1}$ bounds the $n$-manifold
$g_0^{-1}[0,\epsilon]\subset
  \tilde h^{-1}(-\epsilon/2,3\epsilon/2)\subset K_i$ (in green on the figure). 
  This shows that the
  cycle $[S^{n-1}]$ is null-homologous in $K_i$.
\end{proof}

\begin{figure}
\begin{center}
\def\JPicScale{.8}
\input{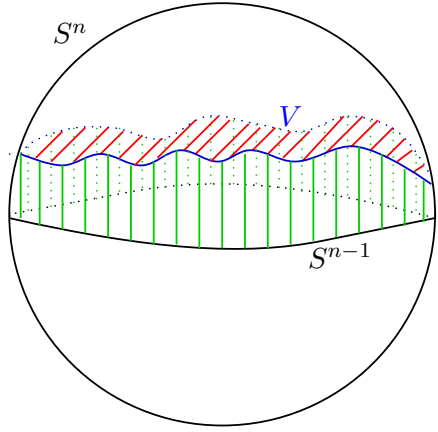}
\caption{The situation in the proof of Lemma \ref{Cech} (illustrated
  on the figure in the case where $n=1$).}
\label{fig-coho}
\end{center}
\end{figure}

\section{Proving that a space is an AR: a topological criterion}\label{sec-miracle}

In this appendix, we establish the ANR criterion that suits our
purpose. Recall from the introduction that a metric space $X$ is
{\it locally $n$-connected} ($LC^n$) if for every $x\in X$ and every open
neighborhood $\calu$ of $x$, there exists an open neighborhood
$\calv\subseteq\calu$ of $x$ such that the inclusion $\calv\hookrightarrow\calu$ is trivial in $\pi_i$ for all $0\leq i\leq n$. A nowhere dense closed subset $Z\subseteq X$ is \emph{locally complementarily $n$-connected} ($LCC^n$) in $X$ if for every $z\in Z$ and every open neighborhood $\calu$ of $z$ in $X$, there exists a smaller neighborhood $\calv\subseteq\calu$ of $z$ in $X$ such that the inclusion $\calv\setminus Z\hookrightarrow\calu\setminus Z$ is trivial in $\pi_i$ for all $0\le i\le n$.

It is a classical fact that a compact $n$-dimensional $LC^n$ metrizable space is an ANR \cite[Theorem V.7.1]{hu}, and we need the following extension. The methods we use are classical, but we could not find this statement in the literature, so we include a proof. We use \cite{hu} as a reference,
but the techniques were established earlier, see
\cite{Kur35,kuratowski2,borsuk,dugundji}.

\begin{theo}\label{miracle}
Let $X$ be a compact metrizable space with $\dim X\leq n$, and let $Z\subset
X$ be a nowhere dense closed subset which is $LCC^n$ in $X$ and such that $X\setminus Z$ is $LC^n$.\\ Then $X$ is an ANR and $Z$ is a $Z$-set in $X$.  \\ If in addition $X\setminus Z$ is
assumed to be contractible, then $X$ is an AR.
\end{theo}

\begin{proof}
The key is to prove the following statement.  \\ \\ {\it Claim.} Let
$Y$ be a compact metrizable space, with $\dim Y\leq n+1$, let $A\subset Y$ be
a closed subset, and let $f:A\to X$ be a map. Then there exists an open
neighborhood $\Oo$ of $A$ in $Y$ and an extension $\tilde f:\Oo\to X$
of $f$ such that $\tilde f(\Oo\setminus A)\subseteq X\setminus
Z$.  \\ \\ \indent We start by explaining how to derive Theorem
\ref{miracle} from the claim. To prove that $Z$ is a $Z$-set, take
$Y=X\times [0,1]$ and $A=X\times\{0\}$, and let $f:A\to X$ be the
identity. Then an extension to a neighborhood produces an
instantaneous homotopy of $X$ off of $Z$, so $Z$ is a $Z$-set (indeed, since $X$
is compact, a neighborhood of $X\times\{0\}$ includes $X\times
[0,\epsilon]$ for some $\epsilon>0$, and we can reparametrize to $X\times [0,1]$). 
  
We now argue that $X$ is $LC^n$, and this will establish that $X$ is
an ANR by \cite[Theorem V.7.1]{hu}. We only need to show that $X$ is $LC^n$ at every point $z\in Z$. Let $z\in Z$ and let $\calu$ be a
neighborhood of $z$ in $X$. Choose a neighborhood $\calv$ of $z$ in $X$ such that
$\calv\setminus Z\hookrightarrow\calu\setminus Z$ is trivial in $\pi_i$ for all
$i\leq n$. Let $f:S^i\to\calv$ be a given map. Using the instantaneous
deformation, we can homotope $f$ to $f':S^i\to \calv\setminus Z$
within $\calv$. But now $f'$ is nullhomotopic within $\calu$ by
assumption.

It remains to prove the claim. Choose an open cover
$\mathcal W$ of $Y\setminus A$ whose multiplicity is at most $n+2$ and
so that the size of the open sets gets small close to $A$. For
example, one could arrange that if $y\in W\in \mathcal W$ then
$\mathrm{diam} W<\frac 12 d(y,A)$ with respect to a fixed metric $d$
on $Y$ (such coverings are called {\it canonical}, see
e.g. \cite[II.11]{hu}). Next, let $N$ be the nerve of $\mathcal W$,
thus $\dim N\leq n+1$. There is a natural topology on $A\cup N$ where
$A$ is closed, $N$ is open, and a neighborhood of $a\in A$ induced by
an open set $\Oo\subset Y$ with $a\in \Oo$ is $\Oo\cap A$
together with the interior of the subcomplex of $N$ spanned by those
$W\in\mathcal W$ contained in $\Oo$. For more details see
e.g. \cite[II.12]{hu}.

Since there is a natural map $\pi:Y\to A\cup N$ that takes $A$ to $A$
and $Y\setminus A$ to $N$ (given by a partition of unity) it
suffices to prove the claim after replacing $Y$ by $A\cup N$.

The extension is constructed by induction on the skeleta of $N$, and
uses the method of e.g. \cite[Theorem V.2.1]{hu}. The only difference
is that we want in addition $\tilde f(\Oo\setminus A)\subseteq
X\setminus Z$.

To
extend $f:A\to X$ to the vertices of $N$, use the assumption that $Z$
is nowhere dense in $X$ to send a vertex $v$ close to some $a\in A$ to a 
point in $X\setminus Z$ close to $f(a)$.
Inductively, suppose $0\leq i\leq n$, 
and $f$ has been extended to the $i$-skeleton of some subcomplex $N_i$
of $N$, in such a way that 
\begin{itemize}
\item $f(N_i)\subseteq X\setminus Z$,
\item $A\cup N_i$ contains a neighborhood of $A$ in $A\cup N$, and 
\item for every $\delta>0$, there exists a neighborhood $\calu_{i,\delta}$ of $A$ in $A\cup N_i$ such that the $f$-image of every $i$-simplex of $N_i$ contained in $\calu_{i,\delta}$ has diameter at most $\delta$.
\end{itemize}

Since $X$ is compact and $Z$ is $LCC^n$ in $X$, there exists $\delta_{i+1}>0$, and a function $r:(0,\delta_{i+1})\to\mathbb{R}_+$, with $r(t)\to 0$ as $t$ decreases to $0$, such that any map $\phi:S^i\to X\setminus Z$ whose image has diameter at most $d<\delta_{i+1}$,
extends to a map $\tilde\phi:B^{i+1}\to X\setminus Z$ whose image has diameter at most
$r(d)$. Now apply this to every $(i+1)$-simplex in $N_i$
such that the image of its boundary has
diameter strictly smaller than $\delta_{i+1}$. When extending to the simplex, always arrange
that the diameter of the image is controlled by the function $r$. This
completes the inductive step.
\end{proof}

\bibliographystyle{amsplain}
\bibliography{AR-bib-final}

\small

\sc \noindent Mladen Bestvina, Department of Mathematics, University of Utah, 155 South 1400 East, JWB 233, Salt Lake City, Utah 84112-0090, United States

\noindent \tt e-mail:bestvina@math.utah.edu
\\
\\
\sc \noindent Camille Horbez, Laboratoire de Math\'ematiques d'Orsay, Univ. Paris-Sud, CNRS, Universit\'e Paris-Saclay, 91405 Orsay, France.

\noindent \tt e-mail:camille.horbez@math.u-psud.fr 
 
\end{document}